\documentclass{amsart}

\usepackage[utf8]{inputenc}
\usepackage{amsfonts}
\usepackage{amsmath}
\usepackage{amsthm}
\usepackage{amssymb}
\usepackage{stmaryrd} 
\usepackage{overpic}[percentage]

\usepackage{hyperref}
\usepackage{color}
\usepackage[dvipsnames]{xcolor}
\usepackage{subcaption}
\usepackage{enumitem}
\usepackage{graphicx}\graphicspath{{./images/}}

\usepackage{tikz}\tikzset{every picture/.style=semithick}
\usetikzlibrary{shapes,arrows,decorations.markings,decorations.pathreplacing}
\tikzset{->-/.style={decoration={
    markings,
    mark=at position #1 with {\arrow{angle 60}}},postaction={decorate}}}

\usepackage[style=numeric]{biblatex}
\newtheorem{theorem}{Theorem}[section]
\newtheorem{corollary}[theorem]{Corollary}
\newtheorem{lemma}[theorem]{Lemma}
\newtheorem{proposition}[theorem]{Proposition}
\theoremstyle{definition}\newtheorem{definition}[theorem]{Definition}
\theoremstyle{remark}\newtheorem{rmk}[theorem]{Remark}
\newtheorem{question}[theorem]{Question}
\newenvironment{remark}{\pushQED{\qed}\rmk}{\popQED\endrmk}
\theoremstyle{plain}\newtheorem*{theorem*}{Theorem}

\title{Generation and Simplicity in the Airplane Rearrangement Group}
\author{Matteo Tarocchi}
\address{Dipartimento di Matematica e Informatica “U. Dini”, Universit\`a di Firenze, Viale Morgagni 67A, I-50134 Firenze, Italy, EU}
\email{\href{mailto:matteo.tarocchi.math@gmail.com}{matteo.tarocchi.math@gmail.com}}

\thanks{This work is part of the author's M.Sc. thesis at the University of Florence. 
The author would like to thank his advisor Francesco Matucci for proposing this problem to him and for all of his insightful guidance and suggestions, and gratefully acknowledges James Belk and Bradley Forrest for helpful conversations and keen comments, and for sharing with the author the source code of their fractal images.
The author wishes to thank an anonymous referee for useful suggestions and improvements to the exposition.}

\begin{document}

\begin{abstract}
We study the group $T_A$ of rearrangements of the Airplane limit space introduced by Belk and Forrest in~\cite{belk2016rearrangement}. We prove that $T_A$ is generated by a copy of Thompson's group $F$ and a copy of Thompson's group $T$, hence it is finitely generated. Then we study the commutator subgroup $[T_A, T_A]$, proving that the abelianization of $T_A$ is isomorphic to $\mathbb{Z}$ and that $[T_A, T_A]$ is simple, finitely generated and acts 2-transitively on the so-called components of the Airplane limit space. Moreover, we show that $T_A$ is contained in $T$ and contains a natural copy of the Basilica rearrangement group $T_B$ studied in~\cite{Belk_2015}.
\end{abstract}

\maketitle

\markboth{Matteo Tarocchi}{
Generation and Simplicity in the Airplane Rearrangement Group}

\section*{Introduction}

In the work \cite{Belk_2015} J. Belk and B. Forrest introduced the Basilica rearrangement group $T_B$ of certain homeomorphisms of the Basilica Julia set (depicted in Figure \ref{fig_B_limit_space}), a \textit{Thompson-like} group that generalizes Thompson's groups $F$ and $T$. These two famous groups are defined as certain groups of orientation-preserving homeomorphisms of $[0,1]$ and $S^1$, respectively, but they have as many equivalent definitions as there are places in which they appear. The group $T$, introduced by Richard Thompson in the 1960's in connection with his work in logic, is the first example of a finitely presented infinite simple group, and it contains natural isomorphic copies of $F$. More about Thompson's groups can be read in \cite{cfp}.

\begin{figure}[h!]\centering
\includegraphics[width=.575\textwidth]{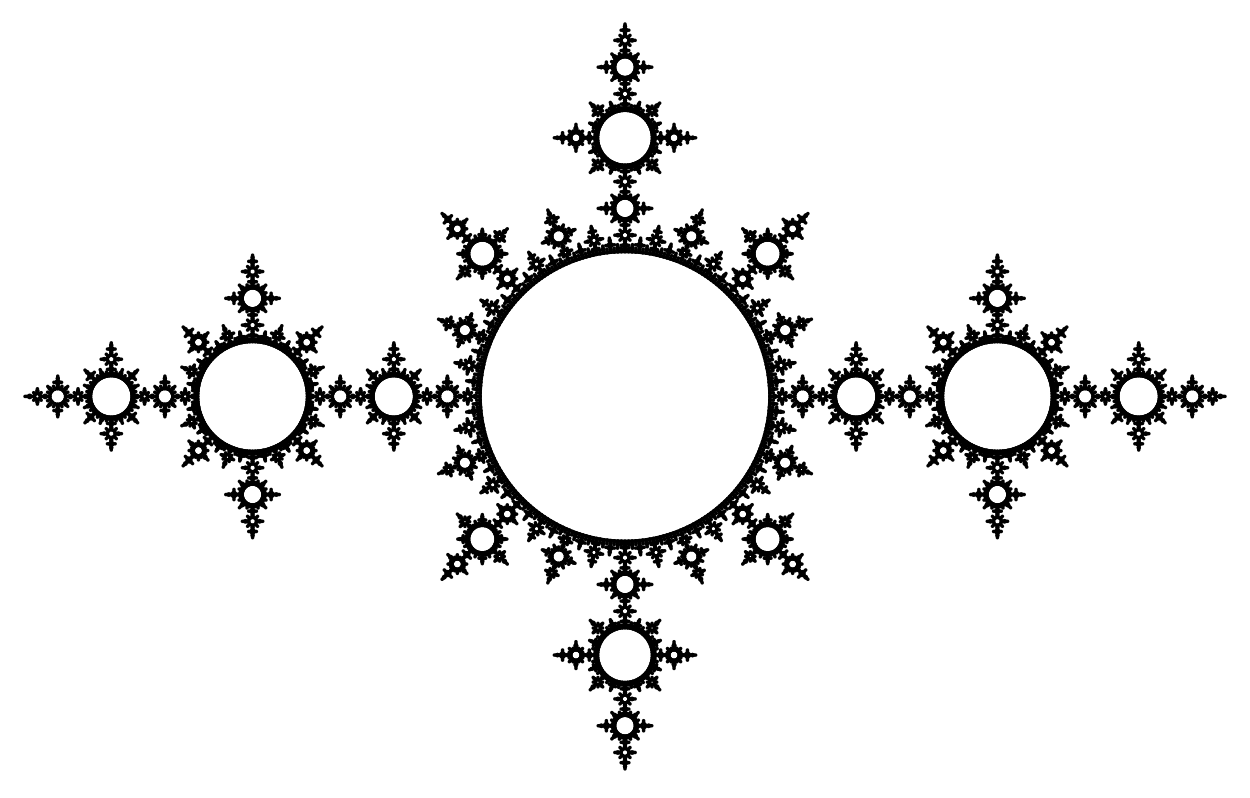}
\caption{The Airplane limit space.}
\label{fig_A_limit_space}
\end{figure}

In the subsequent article \cite{belk2016rearrangement}, Belk and Forrest introduced the family of rearrangement groups of limit spaces, which includes both Thompson's groups $F, T$ and $V$ and the Basilica rearrangement group $T_B$. Each of these groups is associated to a certain fractal, and consists of certain homeomorphisms of the fractal that permute the self-similar cells that it is made up of.

\medskip
In this paper we study the group $T_A$ of rearrangements of the Airplane limit space depicted in Figure \ref{fig_A_limit_space}, which is homeomorphic to a fractal known as the Airplane Julia set. We prove that $T_A$ is generated by natural copies of both Thompson's groups $F$ and $T$, hence $T_A$ is finitely generated. Then we focus our attention on the commutator subgroup of $T_A$, proving in particular that it is simple, finitely generated and infinite index in $T_A$. More precisely, we prove the results collected in the following theorem:

\medskip

\noindent
\textbf{Main Theorem.}
\emph{The group $T_A$ is finitely generated and its commutator subgroup $[T_A, T_A]$ is simple and finitely generated. Moreover, $T_A \simeq [T_A, T_A] \rtimes \mathbb{Z}$.}

\medskip

\noindent

This result shows uncommon behavior in the world of generalized Thompson groups, since $T_A$ is a finitely generated group whose commutator subgroup is infinite index and finitely generated. In several known cases of finitely generated Thompson groups whose commutator subgroup has been studied and shown to be infinite index, the commutator subgroups have also been proved to be infinitely generated: this is, in fact, true for for Thompson's group $F$, the Cleary golden ratio group $F_\tau$ (\cite{burillo2021irrationalslope}), generalized Thompson groups $F_n$ and more generally finitely generated Stein groups over the unit interval $[0,1]$ (see \cite{Stein1992GroupsOP}), since the commutator subgroup in all these groups has support bounded away from $0$ and $1$ and so the standard argument to show infinite generation works out (we will use it too in Proposition~\ref{proposition_E_finitely_generated} for a special subgroup of $[T_A, T_A]$). One notable exception is given by certain topological full groups associated to irreducible shifts of finite type (the groups $\llbracket G_{\varphi_k} \rrbracket$ discussed at the end of \cite{MR3377390}), although these groups are much more ``$V$-like'', whereas $T_A$ is arguably more ``$T$-like''.

\medskip

We also show that $T$ contains an isomorphic copy of $T_A$, and we prove that $T_A$ includes an unexpected natural copy of the Basilica rearrangement group $T_B$ studied in \cite{belk2016rearrangement}. Moreover, we study an infinitely generated subgroup $E$ of the commutator subgroup of $T_A$ and we investigate its transitivity properties, which then extend to both $T_A$ and $[T_A, T_A]$.

\medskip

This paper is organized as follows. Section \ref{section_Thompson} gives a brief introduction to Thompson's groups $F$ and $T$. In Section \ref{section_definitions} we recall the essential definitions of rearrangements and limit spaces from \cite{belk2016rearrangement}. In Section \ref{section_Airplane} we define components, rays and component paths in the Airplane limit space. In Section \ref{section_thompson_in_TA} we exhibit two important natural copies of $F$ and $T$ in $T_A$. In Section \ref{section_generators_Airplane} we prove that $T_A$ is finitely generated. Section \ref{section_commutator} is all about the commutator subgroup $[T_A, T_A]$. In Section \ref{section_circular} we prove that $T$ contains an isomorphic copy of $T_A$. Section \ref{section_E} deals with a specific subgroup $E$ of $T_A$ defined by its action on the extremes of the Airplane limit space. Finally, in Section \ref{section_TB_in_TA} we exhibit a natural copy of $T_B$ contained in $T_A$.

\section{\texorpdfstring{Thompson's groups $F$ and $T$}{Thompson's groups F and T}}\label{section_Thompson}

In this section we briefly introduce Thompson's groups $F$ and $T$, giving their definitions and standard sets of generators. For more details about Thompson's group, we refer the reader to the introductory notes \cite{cfp}.

Consider those partitions of $[0,1]$, such as the one depicted in Figure \ref{fig_subdivision}, that can be obtained by cutting the unit interval in half obtaining $\{[0,\frac{1}{2}], [\frac{1}{2},1]\}$, then cutting one or both of these two intervals in half, and so on, a finite amount of times. These partitions are called \textbf{dyadic subdivisions} of $[0,1]$, and they consist of intervals of the form
\[\left[\frac{a}{2^b},\frac{a+1}{2^b}\right],\]
which are called \textbf{standard dyadic intervals}. Moreover, the extremes of these intervals are \textbf{dyadic points} of $[0,1]$, which means that they belong to $\mathbb{Z}[\frac{1}{2}] \cap [0,1]$.

\begin{figure}[b]\centering
\begin{tikzpicture}
    \draw (0,0) node[below]{$0$} -- (1,0) node[below]{$\frac{1}{4}$} -- (1.5,0) node[below]{$\frac{3}{8}$} -- (2,0) node[below]{$\frac{1}{2}$} -- (3,0) node[below]{$\frac{3}{4}$} -- (4,0) node[below]{$1$};
    \node at (0,0) [circle,fill,inner sep=1.25]{};
    \node at (1,0) [circle,fill,inner sep=1]{};
    \node at (1.5,0) [circle,fill,inner sep=1]{};
    \node at (2,0) [circle,fill,inner sep=1]{};
    \node at (3,0) [circle,fill,inner sep=1]{};
    \node at (4,0) [circle,fill,inner sep=1.25]{};
\end{tikzpicture}
\caption{An example of dyadic subdivision of $[0,1]$.}
\label{fig_subdivision}
\end{figure}
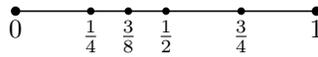

Consider the unit circle $S^1 := \frac{[0,1]}{\{0,1\}}$. By taking the quotient of dyadic subdivisions of $[0,1]$ under the set $\{0,1\}$, we obtain the dyadic subdivisions of $S^1$.

A \textbf{dyadic rearrangement} of the unit interval $[0,1]$ (the unit circle $S^1$) is an orientation-preserving piecewise linear homeomorphism $f: [0,1] \to [0,1]$ ($S^1 \to S^1$) that maps linearly between the intervals of two dyadic subdivisions. \textbf{Thompson's groups} $F$ and $T$ are the groups under composition of the dyadic rearrangements of the unit interval $[0,1]$ and the unit circle $S^1$, respectively.

The elements of Thompson's group $F$ are specified by a pair of dyadic subdivisions of $[0,1]$ with the same number of dyadic intervals, such as those in Figure \ref{fig_F_gen}. By this we mean that the $n$-th interval of the first subdivision is mapped linearly into the $n$-th interval of the second subdivision. In a similar fashion, the elements of Thompson's group $T$ are specified by a pair of dyadic subdivisions of $S^1$ with the same number of dyadic intervals, along with certain colorations of the dyadic intervals for the two subdivisions, such as those in Figure \ref{fig_T_gen}. By this we mean that an interval of the first subdivision is mapped linearly to the interval of the second subdivision that has the same color (or, equivalently, the same letter). Note that a single color (or a single letter) would suffice, since the elements of $T$ are orientation-preserving.

\begin{figure}\centering
\begin{subfigure}{.475\textwidth}
\centering
\begin{tikzpicture}[scale=.6]
    \draw (0,0.5) node{};
    
    \draw (0,0) node[below] {$0$} -- (1,0) node[below] {$\frac{1}{4}$} -- (2,0) node[below] {$\frac{1}{2}$} -- (4,0) node[below] {$1$};
    \node at (0,0) [circle,fill,inner sep=1]{};
    \node at (1,0) [circle,fill,inner sep=1]{};
    \node at (2,0) [circle,fill,inner sep=1]{};
    \node at (4,0) [circle,fill,inner sep=1]{};
    
    \draw[-to] (4.55,0) -- node[midway,above]{$X_0$} (4.95,0);
    
    \draw (11/2,0) node[below] {$0$} -- (15/2,0) node[below] {$\frac{1}{2}$} -- (17/2,0) node[below] {$\frac{3}{4}$} -- (19/2,0) node[below] {$1$};
    \node at (11/2,0) [circle,fill,inner sep=1]{};
    \node at (15/2,0) [circle,fill,inner sep=1]{};
    \node at (17/2,0) [circle,fill,inner sep=1]{};
    \node at (19/2,0) [circle,fill,inner sep=1]{};
\end{tikzpicture}
\end{subfigure}
\hfill
\begin{subfigure}{.475\textwidth}
\centering
\begin{tikzpicture}[scale=.6]
    \draw (0,0.5) node{};
    
    \draw (0,0) node[below] {$0$} -- (2,0) node[below] {$\frac{1}{2}$} -- (2.5,0) node[below] {$\frac{5}{8}$} -- (3,0) node[below] {$\frac{3}{4}$} -- (4,0) node[below] {$1$};
    \node at (0,0) [circle,fill,inner sep=1]{};
    \node at (2,0) [circle,fill,inner sep=1]{};
    \node at (2.5,0) [circle,fill,inner sep=1]{};
    \node at (3,0) [circle,fill,inner sep=1]{};
    \node at (4,0) [circle,fill,inner sep=1]{};
    
    \draw[-to] (4.55,0) -- node[midway,above]{$X_1$} (4.95,0);
    
    \draw (11/2,0) node[below] {$0$} -- (15/2,0) node[below] {$\frac{1}{2}$} -- (17/2,0) node[below] {$\frac{3}{4}$} -- (9,0) node[below] {$\frac{7}{8}$} -- (19/2,0) node[below] {$1$};
    \node at (11/2,0) [circle,fill,inner sep=1]{};
    \node at (15/2,0) [circle,fill,inner sep=1]{};
    \node at (17/2,0) [circle,fill,inner sep=1]{};
    \node at (9,0) [circle,fill,inner sep =1]{};
    \node at (19/2,0) [circle,fill,inner sep=1]{};
\end{tikzpicture}
\end{subfigure}
\caption{The generators $X_0$ and $X_1$ of Thompson's group $F$.}
\label{fig_F_gen}
\end{figure}
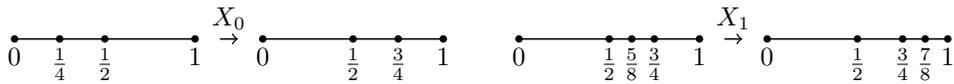

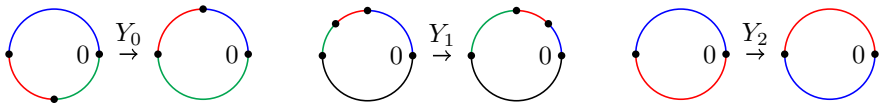
\begin{figure}[b]\centering
\begin{subfigure}{.32\textwidth}
\centering
\begin{tikzpicture}[scale=.6]
    \draw[blue] (1,0) node[black,left]{$0$} arc [radius=1, start angle=0, end angle=180] node[midway,above]{\footnotesize A};
    \draw[red] (-1,0) arc [radius=1, start angle=180, end angle=270] node[midway,below left]{\footnotesize B};
    \draw[Green] (0,-1) arc [radius=1, start angle=270, end angle=360] node[midway,below right]{\footnotesize C};
    
    \node at (1,0) [circle,fill,inner sep=1]{};
    \node at (0,-1) [circle,fill,inner sep=1]{};
    \node at (-1,0) [circle,fill,inner sep=1]{};
    
    \draw[-to] (1.45,0) -- node[midway,above]{$Y_0$} (1.85,0);
    
    \draw[blue] (4.3,0) node[black,left]{$0$} arc [radius=1, start angle=0, end angle=90] node[midway,above right]{\footnotesize A};
    \draw[red] (3.3,1) arc [radius=1, start angle=90, end angle=180] node[midway,above left]{\footnotesize B};
    \draw[Green] (2.3,0) arc [radius=1, start angle=180, end angle=360] node[midway,below]{\footnotesize C};
    
    \node at (4.3,0) [circle,fill,inner sep=1]{};
    \node at (3.3,1) [circle,fill,inner sep=1]{};
    \node at (2.3,0) [circle,fill,inner sep=1]{};
\end{tikzpicture}
\end{subfigure}
\begin{subfigure}{.32\textwidth}
\centering
\begin{tikzpicture}[scale=.6]
    \draw[blue] (1,0) node[black,left]{$0$} arc [radius=1, start angle=0, end angle=90] node[midway,above right]{\footnotesize A};
    \draw[red] (0,1) arc [radius=1, start angle=90, end angle=135] node[midway,above]{\footnotesize B};
    \draw[Green] (-1,0) arc [radius=1, start angle=180, end angle=135] node[midway,left]{\footnotesize C};
    \draw (-1,0) arc [radius=1, start angle=180, end angle=360] node[midway,below]{\footnotesize D};
    
    \node at (1,0) [circle,fill,inner sep=1]{};
    \node at (0,1) [circle,fill,inner sep=1]{};
    \node at (-1,0) [circle,fill,inner sep=1]{};
    \node at (-0.70711,0.70711) [circle,fill,inner sep=1]{};
    
    \draw[-to] (1.45,0) -- node[midway,above]{$Y_1$} (1.85,0);
    
    \draw[blue] (4.3,0) node[black,left]{$0$} arc [radius=1, start angle=0, end angle=45] node[midway,right]{\footnotesize A};
    \draw[red] (3.3,1) arc [radius=1, start angle=90, end angle=45] node[midway,above]{\footnotesize B};
    \draw[Green] (2.3,0) arc [radius=1, start angle=180, end angle=90] node[midway,left]{\footnotesize C};
    \draw (2.3,0) arc [radius=1, start angle=180, end angle=360] node[midway,below]{\footnotesize D};
    
    \node at (4.3,0) [circle,fill,inner sep=1]{};
    \node at (3.3,1) [circle,fill,inner sep=1]{};
    \node at (2.3,0) [circle,fill,inner sep=1]{};
    \node at (4.00711,0.70711) [circle,fill,inner sep=1]{};
\end{tikzpicture}
\end{subfigure}
\begin{subfigure}{.32\textwidth}
\centering
\begin{tikzpicture}[scale=.6]
    \draw[blue] (1,0) node[black,left]{$0$} arc [radius=1, start angle=0, end angle=180] node[midway,above]{\footnotesize A};
    \draw[red] (-1,0) arc [radius=1, start angle=180, end angle=360] node[midway,below]{\footnotesize B};
    
    \node at (1,0) [circle,fill,inner sep=1]{};
    \node at (-1,0) [circle,fill,inner sep=1]{};
    
    \draw[-to] (1.45,0) -- node[midway,above]{$Y_2$} (1.85,0);
    
    \draw[red] (4.3,0) node[black,left]{$0$} arc [radius=1, start angle=0, end angle=180] node[midway,above]{\footnotesize B};
    \draw[blue] (2.3,0) arc [radius=1, start angle=180, end angle=360] node[midway,below]{\footnotesize A};
    
    \node at (4.3,0) [circle,fill,inner sep=1]{};
    \node at (2.3,0) [circle,fill,inner sep=1]{};
\end{tikzpicture}
\end{subfigure}
\caption{The three generators of Thompson's group $T$.}
\label{fig_T_gen}
\end{figure}

The group $F$ is generated by the two elements $X_0$ and $X_1$ depicted in Figure \ref{fig_F_gen}, while $T$ is generated by $Y_0, Y_1$ and $Y_2$ depicted in Figure \ref{fig_T_gen}. Further properties of $F$ and $T$ can be found in \cite{belk2007thompsons, burillo, cfp}.

\section{Limit spaces and rearrangements}\label{section_definitions}

Limit spaces of replacement systems and their rearrangements were introduced in \cite{belk2016rearrangement}, which goes in much more details than we will get to do. In this section we briefly describe these notions, introducing the Airplane limit space along the way.

\subsection{Replacement systems and limit spaces}\label{subsection_limit_spaces}

Essentially, a \textbf{replacement system} consists of a \textbf{base graph} $\Gamma$ colored by the set of colors $Col$, along with a \textbf{replacement graph} $R_c$ for each color $c \in Col$. Figure \ref{fig_replacement_A} depicts the so called Airplane replacement system, denoted by $\mathcal{A}$. We can expand the base graph $\Gamma$ by replacing one of its edges $e$ with the replacement graph $R_c$ indexed by the color $c$ of $e$, as exemplified in Figure \ref{fig_exp_A}. The graph resulting from this process of replacing one edge with the appropriate replacement graph is called a \textbf{simple expansion}. Simple expansions can be repeated any finite amount of times, which generate the so-called \textbf{expansions} of the replacement system, such as the one in Figure \ref{fig_exp_A_generic}.

Note that each edge of an expansion corresponds to the unique finite sequence of edges ``converging'' from the base graph. As an example, consider the leftmost edge of Figure \ref{fig_exp_A_generic}. In order to identify this edge, one must first restrict their attention to the leftmost blue edge $e_0$ of the base graph; then one expands this edge with the blue replacement graph and restrict their attention to the leftmost edge $e_1$ that has thus been generated; finally, one expands this edge and consider the leftmost edge $e_2$ generated by the last expansion, which is precisely the edge that we were looking for. In this sense, we say that the leftmost edge of Figure \ref{fig_exp_A_generic} corresponds to the sequence $e_0 e_1 e_2$. For more precise definitions we refer the reader to Section 1.1 of \cite{belk2016rearrangement}.

\begin{figure}\centering
\begin{subfigure}{.4\textwidth}
\centering
\begin{tikzpicture}[scale=.8]
\draw[->-=.5,blue] (-0.5,0) -- (-2,0) node[black,circle,fill,inner sep=1.25]{}; \draw[->-=.5,blue] (0.5,0) -- (2,0) node[black,circle,fill,inner sep=1.25]{};
\draw[->-=.5,red] (0.5,0) node[circle,fill,inner sep=1.25]{} to[out=90,in=90,looseness=1.7] (-0.5,0);
\draw[->-=.5,red] (-0.5,0) node[black,circle,fill,inner sep=1.25]{} to[out=270,in=270,looseness=1.7] (0.5,0) node[black,circle,fill,inner sep=1.25]{};
\end{tikzpicture}
\caption{The base graph.}
\label{fig_A_base}
\end{subfigure}
\begin{subfigure}{.5\textwidth}
\centering
\begin{subfigure}{.45\textwidth}
\centering
\begin{tikzpicture}[scale=.8]
\draw[->-=.5,draw=red] (0,0) node[left]{$v_i$} node[black,circle,fill,inner sep=1.25]{} -- (0,2.1) node[left]{$v_t$} node[black,circle,fill,inner sep=1.25]{};

\draw[-stealth] (0.35,1.05) -- (0.65,1.05);

\draw[->-=.5,red] (1,0) node[black,circle,fill,inner sep=1.25]{} node[black,left]{$v_i$} -- (1,1.05);
\draw[->-=.5,red] (1,1.05) -- (1,2.1) node[black,circle,fill,inner sep=1.25]{} node[black,left]{$v_t$};
\draw[->-=.5,blue] (1,1.05) node[black,circle,fill,inner sep=1.25]{} -- (2,1.05) node[black,circle,fill,inner sep=1.25]{};
\end{tikzpicture}
\end{subfigure}
\begin{subfigure}{.45\textwidth}
\centering
\begin{tikzpicture}[scale=.8]
\draw[->-=.5,draw=blue] (0,0) node[black,left]{$v_i$} node[black,circle,fill,inner sep=1.25]{} -- (0,2.1) node[black,left]{$v_t$} node[black,circle,fill,inner sep=1.25]{};

\draw[-stealth] (0.35,1.05) -- (0.65,1.05);

\draw[->-=.5,blue] (1.35,0.7) -- (1.35,0) node[black,circle,fill,inner sep=1.25]{} node[black,left]{$v_i$};
\draw[->-=.5,blue] (1.35,1.4) -- (1.35,2.1) node[black,circle,fill,inner sep=1.25]{} node[black,left]{$v_t$};
\draw[->-=.5,red] (1.35,1.4) to[out=180,in=180,looseness=1.7] (1.35,0.7);
\draw[->-=.5,red] (1.35,0.7) node[black,circle,fill,inner sep=1.25]{} to[out=0,in=0,looseness=1.7] (1.35,1.4) node[black,circle,fill,inner sep=1.25]{};
\end{tikzpicture}
\end{subfigure}
\caption{The two replacement rules: $e \to R_{red}$ if $e$ is red, and $e \to R_{blue}$ if $e$ is blue.}
\label{fig_A_replacement_rule}
\end{subfigure}
\caption{The Airplane replacement system $\mathcal{A}$.}
\label{fig_replacement_A}
\end{figure}

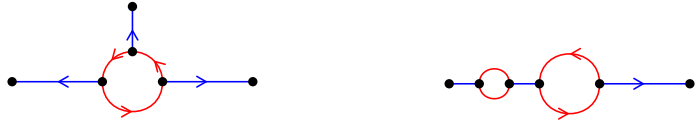
\begin{figure}\centering
\begin{subfigure}{.45\textwidth}\centering
\begin{tikzpicture}[scale=.8]
\draw[->-=.5,blue] (-0.5,0) -- (-2,0) node[black,circle,fill,inner sep=1.25]{}; \draw[->-=.5,blue] (0.5,0) -- (2,0) node[black,circle,fill,inner sep=1.25]{};
\draw[->-=.5,red] (0.5,0) to[out=90,in=0] (0,0.5);
\draw[->-=.5,blue] (0,0.5) -- (0,1.25) node[black,circle,fill,inner sep=1.25]{};
\draw[->-=.5,red] (0,0.5) node[black,circle,fill,inner sep=1.25]{} to[out=180,in=90] (-0.5,0);
\draw[->-=.5,red] (-0.5,0) node[black,circle,fill,inner sep=1.25]{} to[out=270,in=270,looseness=1.7] (0.5,0) node[black,circle,fill,inner sep=1.25]{};
\end{tikzpicture}
\end{subfigure}
\begin{subfigure}{.45\textwidth}
\centering
\begin{tikzpicture}[scale=.8]
\draw[blue] (-0.5,0) -- (-1,0); \draw[blue] (-2,0) node[black,circle,fill,inner sep=1.25]{} -- (-1.5,0);
\draw[red] (-1,0) to[out=90,in=90,looseness=1.7] (-1.5,0);
\draw[red] (-1.5,0) node[black,circle,fill,inner sep=1.25]{} to[out=270,in=270,looseness=1.7] (-1,0) node[black,circle,fill,inner sep=1.25]{};
\draw[->-=.5,blue] (0.5,0) -- (2,0) node[black,circle,fill,inner sep=1.25]{};
\draw[->-=.5,red] (0.5,0) to[out=90,in=90,looseness=1.7] (-0.5,0);
\draw[->-=.5,red] (-0.5,0) node[black,circle,fill,inner sep=1.25]{} to[out=270,in=270,looseness=1.7] (0.5,0) node[black,circle,fill,inner sep=1.25]{};
\node at (0,1.25) {};
\end{tikzpicture}
\end{subfigure}
\caption{Two simple expansions of the base graph of the Airplane replacement system $\mathcal{A}$.}
\label{fig_exp_A}
\end{figure}

\begin{figure}\centering
\begin{tikzpicture}[scale=1.65]
\draw[blue] (-0.5,0) -- (-2,0);
\draw[blue] (0.5,0) -- (2,0);
\draw[blue] (0,-0.5) -- (0,-1.25);
\draw[blue] (0.35,-0.35) -- (0.75,-0.75);
\draw[blue] (0.55,-0.55) -- (0.7,-0.4);
\draw[blue] (1.25,0) -- (1.25,0.75);
\draw[blue] (1.25,0) -- (1.625,0.375);
\draw[blue] (-1.25,0) -- (-1.25,0.75);
\draw[blue] (-1.25,0.5) -- (-1.05,0.5);
\draw[blue] (-1.75,0) -- (-1.75,-0.2);
\draw[red] (0,0) circle (0.5);
\draw[red,fill=white] (-1.25,0) circle (0.25);
\draw[red,fill=white] (1.25,0) circle (0.25);
\draw[red,fill=white] (0.55,-0.55) circle (0.0625);
\draw[red,fill=white] (-1.25,0.5) circle (0.0625);
\draw[red,fill=white] (-1.11875,0.5) circle (0.025);
\draw[red,fill=white] (-1.75,0) circle (0.0625);
\end{tikzpicture}
\caption{A generic expansion of the Airplane replacement system $\mathcal{A}$.}
\label{fig_exp_A_generic}
\end{figure}
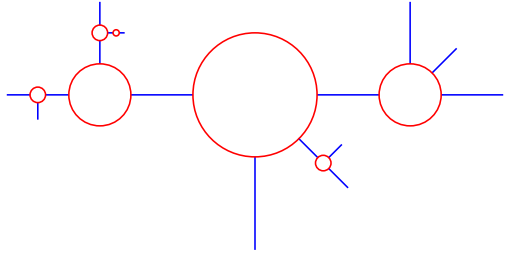

Consider the \textbf{full expansion sequence}, which is the sequence of graphs obtained by replacing, at each step, every edge with the appropriate replacement graph, starting from the base graph. Figure \ref{fig_fullexp_A} shows the first graphs (except for the base graph) of the full expansion sequence for the Airplane replacement system $\mathcal{A}$.
If a replacement system satisfies certain simple properties (which $\mathcal{A}$ satisfy), then we can define the \textbf{limit space}, which is essentially the limit of the full expansion sequence of the replacement system (Definition 1.8 and Proposition 1.9 of \cite{belk2016rearrangement}). Keeping in mind that finite sequences of edges correspond to edges of expansions, the limit space should be thought of as the set of infinite sequences of edges modulo an equivalence relation that ``glues'' certain sequences together.
For example, if one expands the top red edge of the base graph and then keeps expanding (infinitely many times) the leftmost red edge, or if one expands the left blue edge of the base graph and then keeps expanding the rightmost blue edge, then one has found the same point with two distinct sequences, which must then be glued together;
intuitively, this point corresponds to the second vertex of the base graph, from the left.
Again, for more details we refer to Definitions 1.6 and 1.7 of \cite{belk2016rearrangement}. In particular, the the Airplane limit space is the space depicted in Figure \ref{fig_A_limit_space}. By Proposition 9 of \cite{belk2016rearrangement}, this is a compact and metrizable topological space.

It is worth mentioning that our Airplane limit space (and limit spaces of replacement systems in general) is a \textit{topological space}, whereas the Airplane Julia set is a fractal embedded in the euclidean plane. This Julia set is only one of the infinitely many quadratic Julia sets corresponding to the complex functions $f(z) = z^2 - p$ for $p$ belonging to the same interior component of the Mandelbrot set as $p = 1.755$, and these are all homeomorphic as topological spaces, although they are different as Julia sets and have different metric properties. Here we will only be treating the Airplane fractal as a topological space.

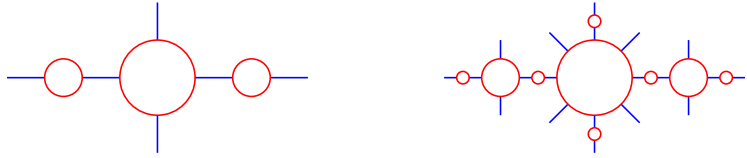
\begin{figure}\centering
\begin{subfigure}{.45\textwidth}
\centering
\begin{tikzpicture}[scale=1]
\draw[blue] (-2,0) -- (2,0);
\draw[blue] (0,1) -- (0,-1);
\draw[red,fill=white] (0,0) circle (0.5);
\draw[red,fill=white] (-1.25,0) circle (0.25);
\draw[red,fill=white] (1.25,0) circle (0.25);
\end{tikzpicture}
\end{subfigure}
\begin{subfigure}{.45\textwidth}
\centering
\begin{tikzpicture}[scale=1]
\draw[blue] (-2,0) -- (2,0);
\draw[blue] (0,1) -- (0,-1);
\draw[blue] (-1.25,0.5) -- (-1.25,-0.5);
\draw[blue] (1.25,0.5) -- (1.25,-0.5);
\draw[blue] (-0.6,-0.6) -- (0.6,0.6);
\draw[blue] (-0.6,0.6) -- (0.6,-0.6);
\draw[red,fill=white] (0,0) circle (0.5);
\draw[red,fill=white] (-1.25,0) circle (0.25);
\draw[red,fill=white] (1.25,0) circle (0.25);
\draw[red,fill=white] (0.75,0) circle (0.083);
\draw[red,fill=white] (1.75,0) circle (0.083);
\draw[red,fill=white] (-0.75,0) circle (0.083);
\draw[red,fill=white] (-1.75,0) circle (0.083);
\draw[red,fill=white] (0,0.75) circle (0.083);
\draw[red,fill=white] (0,-0.75) circle (0.083);
\end{tikzpicture}
\end{subfigure}
\caption{The beginning of the full expansion sequence for $\mathcal{A}$.}
\label{fig_fullexp_A}
\end{figure}

\subsection{Cells and rearrangements}\label{subsection_rearrangements}

Intuitively, a cell $\chi(e)$ of a limit space corresponds to the edge $e$ of some expansion, along with everything that appears from that edge in later expansions. More precisely, if the edge $e$ corresponds to the finite sequence $e_0 \cdots e_k$, then the cell $\chi(e)$ is the subset of a limit space consisting of those infinite sequences of edges that start with $e_0 \cdots e_k$. For example, Figure \ref{fig_cells_A} shows examples of cells in $\mathcal{A}$. Moreover, we say that the cell $\chi(e)$ is colored by $c$ if the edge $e$ is colored by $c$.

\begin{figure}\centering
\includegraphics[width=.425\textwidth]{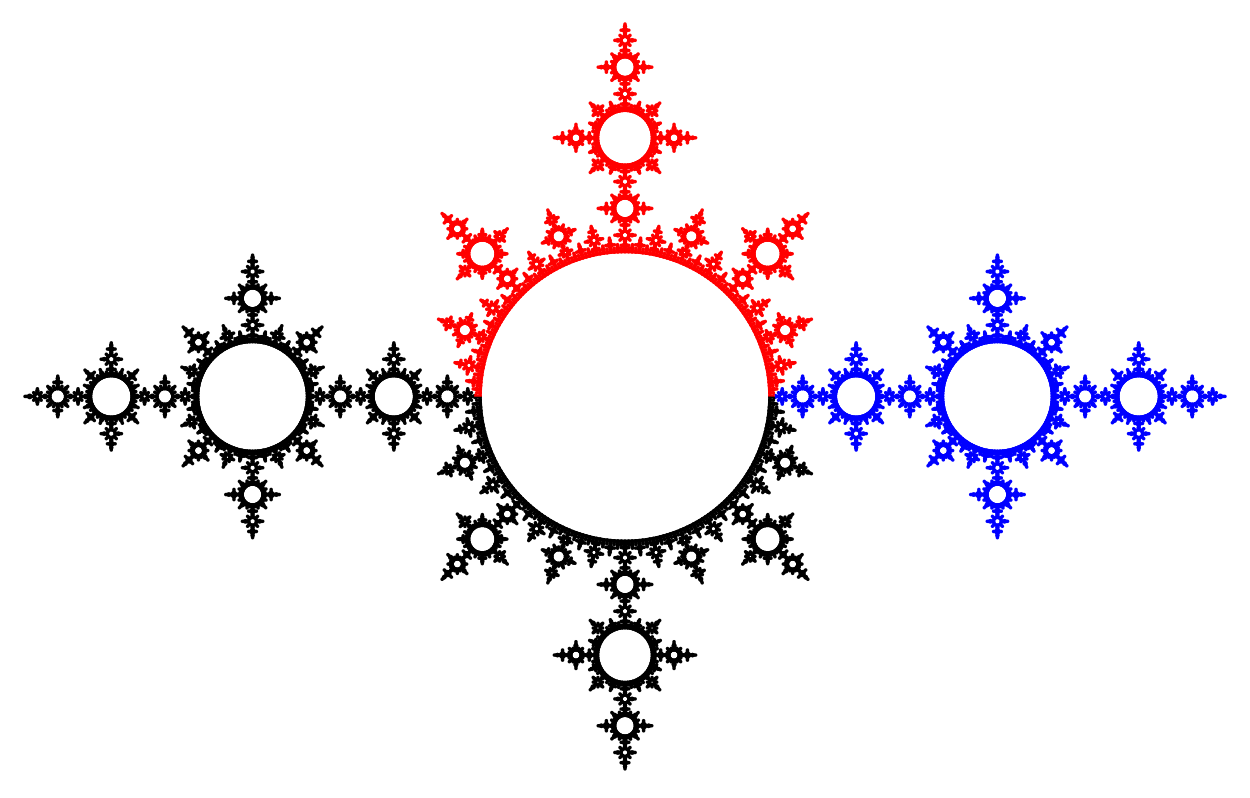}
\caption{The two types of cells in the Airplane replacement system $\mathcal{A}$, distinguished by the color of the generating edge.}
\label{fig_cells_A}
\end{figure}

There are different \textit{types} of cells $\chi(e)$, distinguished by two aspects of the generating edge $e$: its color and whether it is a loop or not. It is not hard to see that there is a \textbf{canonical homeomorphism} between any two cells of the same type. More precisely, if the two cells correspond to the edges identified by the sequences $e_0 \cdots e_k$ and $f_0 \cdots f_l$, then the homeomorphism maps each infinite sequence $e_0 \cdots e_k w$ to $f_0 \cdots f_l w$. A canonical homeomorphism between two cells can essentially be thought of as a transformation that maps the first cell ``rigidly'' to the second.
More details are given in Section 1.3 of \cite{belk2016rearrangement}. Note that there are only two types of cells in the Airplane limit space, the red one and the blue one (depicted in Figure \ref{fig_cells_A}), because no edge of any expansion is ever a loop.

\begin{definition}
A \textbf{cellular partition} of the limit space $X$ is a cover of $X$ by finitely many cells whose interiors are disjoint.
\end{definition}

Note that there is a natural bijection between the set of expansions of a replacement system and the set of cellular partitions, which consists of mapping each edge $e$ of the expansion to the cell $\chi(e)$.

\begin{definition}[Definition 1.14(2) of \cite{belk2016rearrangement})]
A homeomorphism $f: X \to X$ is called a \textbf{rearrangement} of $X$ if there exists a cellular partition $\mathcal{P}$ of $X$ such that $f$ restricts to a canonical homeomorphism on each cell of $\mathcal{P}$.
\end{definition}

It can be proved that the rearrangements of a limit space $X$ form a group under composition, called the \textbf{rearrangement group} of $X$ (Proposition 1.16 of \cite{belk2016rearrangement}). In particular, the rearrangement group of the Airplane limit space is denoted by $T_A$.

\begin{figure}\centering
\vspace{.2cm}
\begin{minipage}{.345\textwidth}\centering
\begin{overpic}[width=\textwidth]{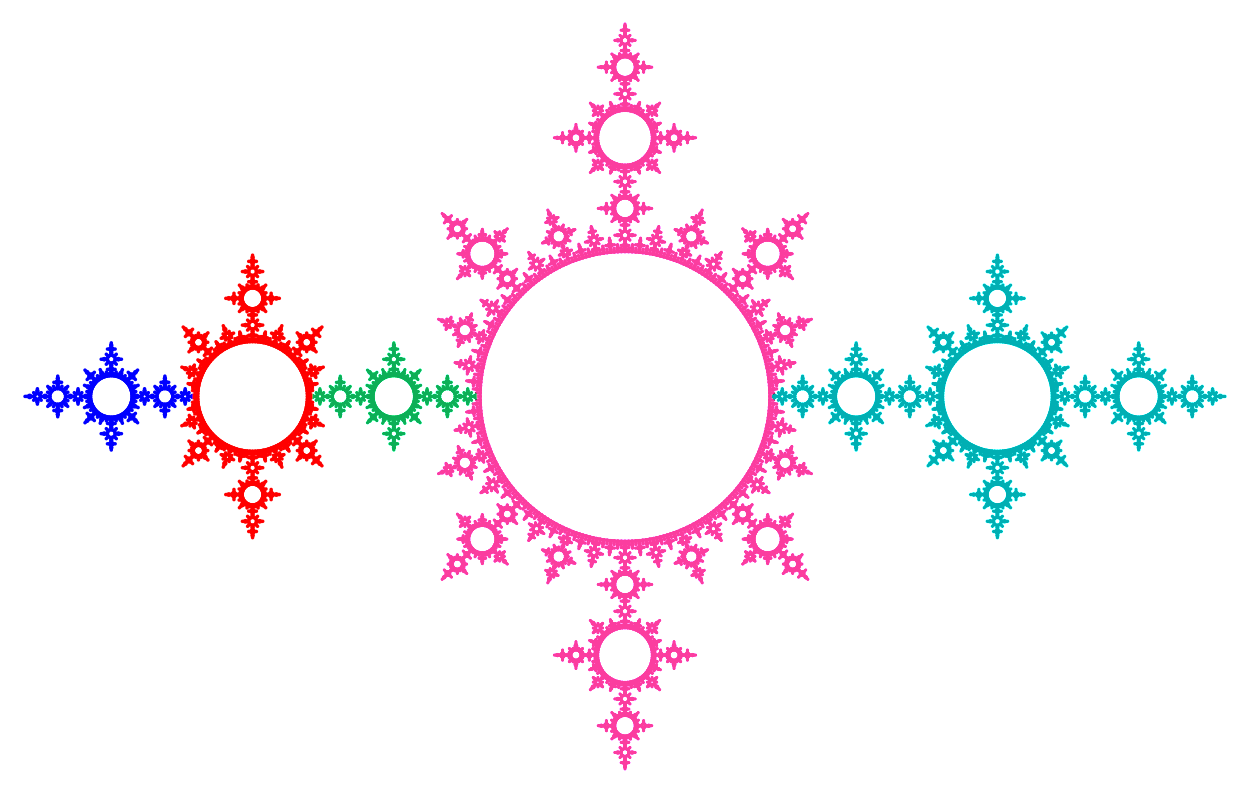}
    \put (6.5,38.5) {\footnotesize\textcolor{blue}{A}}
    \put (18,46) {\footnotesize\textcolor{red}{B}}
    \put (28.5,38.5) {\footnotesize\textcolor{Green}{C}}
    \put (47.5,64) {\footnotesize\textcolor{Magenta}{D}}
    \put (77.5,46) {\footnotesize\textcolor{TealBlue}{E}}
\end{overpic}
\end{minipage}%
\begin{minipage}{.05\textwidth}\centering
\begin{tikzpicture}[scale=1]
    \draw[-to] (0,0) -- (.3,0);
\end{tikzpicture}
\end{minipage}%
\begin{minipage}{.345\textwidth}\centering
\begin{overpic}[width=\textwidth]{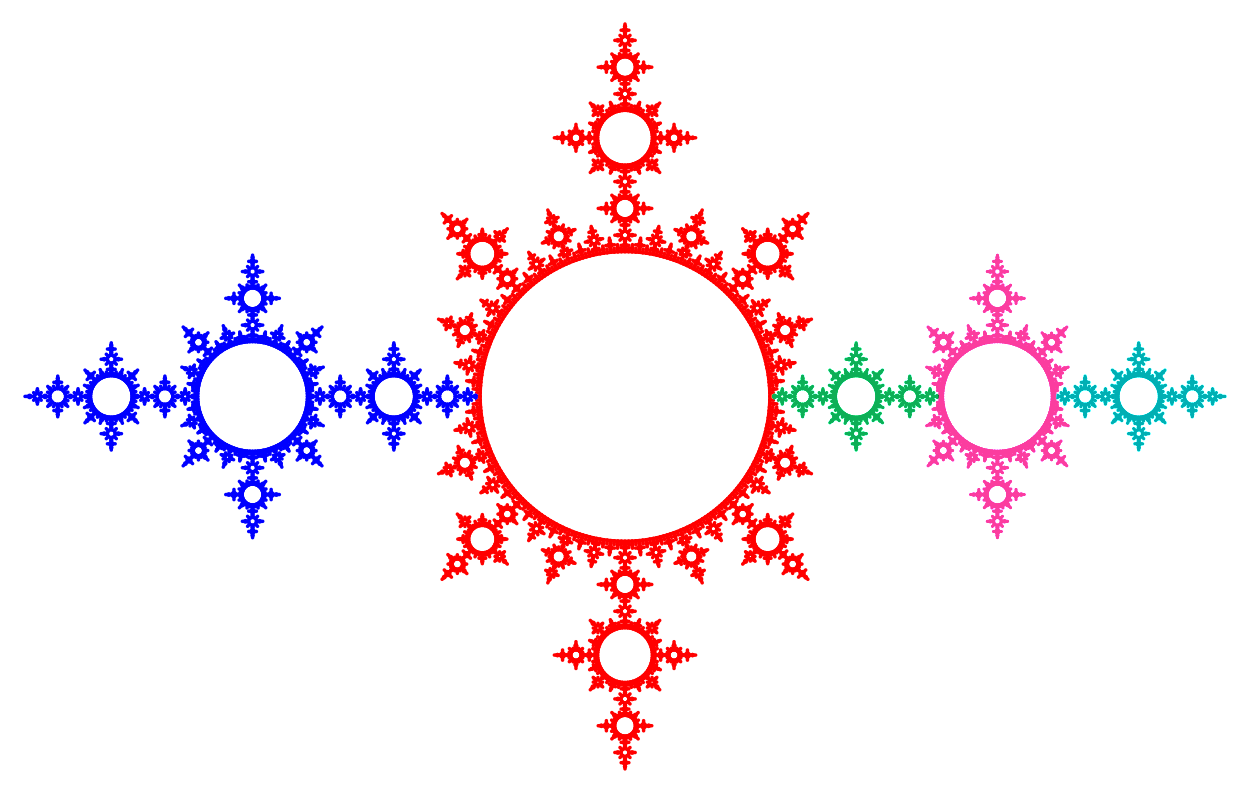}
    \put (17.5,46) {\footnotesize\textcolor{blue}{A}}
    \put (47.5,64) {\footnotesize\textcolor{red}{B}}
    \put (66,38.5) {\footnotesize\textcolor{Green}{C}}
    \put (77.5,46) {\footnotesize\textcolor{Magenta}{D}}
    \put (88.5,38.5) {\footnotesize\textcolor{TealBlue}{E}}
\end{overpic}
\end{minipage}\\
\vspace*{10pt}
\begin{tikzpicture}[scale=1]
    \draw[blue] (-2,0) -- node[midway,above]{\footnotesize A} (-1.5,0);
    \draw[red] (-1.25,0) circle (0.25);
    \draw[red] (-1.25,.25) node[above]{\footnotesize B};
    \draw[Green] (-1,0) -- (-0.5,0) node[midway,above]{\footnotesize C};
    \draw[Magenta] (0,0) circle (0.5);
    \draw[Magenta] (0,.5) node[above]{\footnotesize D};
    \draw[TealBlue] (0.5,0) -- (2,0) node[midway,above]{\footnotesize E};
    
    \draw[-to] (2.35,0) -- (2.65,0);
    
    \draw[blue] (3,0) -- (4.5,0) node[midway,above]{\footnotesize A};
    \draw[red] (5,0) circle (0.5);
    \draw[red] (5,.5) node[above]{\footnotesize B};
    \draw[Green] (5.5,0) -- (6,0) node[midway,above]{\footnotesize C};
    \draw[Magenta] (6.25,0) circle (0.25);
    \draw[Magenta] (6.25,.25) node[above]{\footnotesize D};
    \draw[TealBlue] (6.5,0) -- (7,0) node[midway,above]{\footnotesize E};
\end{tikzpicture}
\caption{A rearrangement of the Airplane limit space, along with a graph pair diagram that represents it.}
\label{fig_action_alpha}
\end{figure}

\phantomsection%
\label{graph_pair_diagrams}%
Similarly to how dyadic rearrangements (the elements of Thompson's groups) are specified by certain pairs of dyadic subdivisions, rearrangements of a limit space are specified by certain graph isomorphisms between expansions of the replacement systems, called \textbf{graph pair diagrams} (Section 1.4 of \cite{belk2016rearrangement}). For example, the rearrangement of the Airplane limit space depicted in Figure \ref{fig_action_alpha} is specified by the graph isomorphism depicted in the same figure, where the colors mean that each edge of the domain graph is mapped to the edge of the same color in the range graph.

Graph pair diagrams can be expanded by expanding an edge in the domain graph and its image in the range graph, resulting in a graph pair diagram that represents the same rearrangements. It is important to note that, for each rearrangement, there exists a unique \textbf{reduced} graph pair diagram, where reduced means that it is not the result of an expansion of any other graph pair diagram.

One may note that the graph isomorphism depicted in Figure \ref{fig_action_alpha} is not valid because the orientation of the \textcolor{Green}{green} edge is reversed. In truth, it is not hard to see that expanding that edge in both domain and range graphs provides a valid graph pair diagram. This holds in general for blue edges, because there is a graph automorphism of the blue replacement graph that switches the initial and terminal vertices, so the expansion of a blue edge is independent from its orientation. In practice, this means that we do not need to keep track of the orientation of blue edges for graph pair diagrams. However, this is not true for red edges.

\subsection{The rearrangement group of the Basilica limit space}

As another example of rearrangement group, consider the Basilica limit space (Figure \ref{fig_B_limit_space}), resulting from the (monochromatic) Basilica replacement system (Figure \ref{fig_replacement_B}) and whose rearrangement group $T_B$ is the object of the study of Belk and Forrest in \cite{Belk_2015}. In particular, Belk and Forrest proved that $T_B$ is generated by the four elements depicted in Figure \ref{fig_TB_generators}, and that the commutator subgroup $[T_B, T_B]$ is simple. In Section \ref{section_TB_in_TA} we will prove that $T_A$ contains a natural copy of this group $T_B$.

\begin{figure}\centering
\includegraphics[width=.47\textwidth]{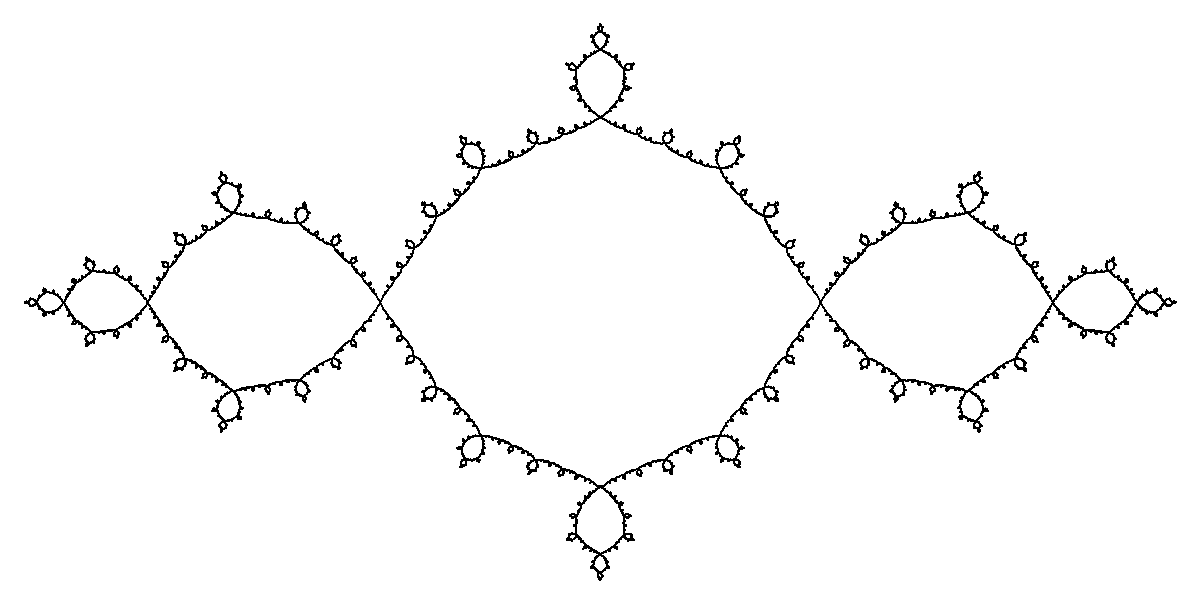}
\caption{The Basilica limit space $B$.}
\label{fig_B_limit_space}
\end{figure}

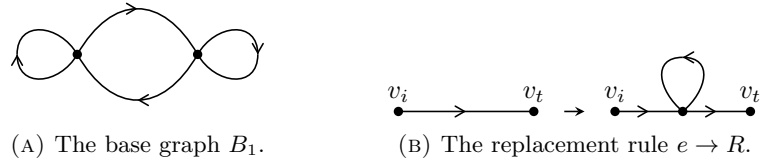
\begin{figure}\centering
\begin{subfigure}[b]{.45\textwidth}\centering
\begin{tikzpicture}[scale=.8]
\draw[->-=.5] (0,0) node[circle,fill,inner sep=1.25]{} to[out=60,in=120,looseness=1.5] (2,0);
\draw[->-=.5] (2,0) node[circle,fill,inner sep=1.25]{} to[out=240,in=300,looseness=1.5] (0,0);
\draw (0,0) to[in=270,out=240,looseness=1.5] (-1,0); \draw[->-=0] (-1,0) to[in=120,out=90,looseness=1.5] (0,0);
\draw (2,0) to[in=90,out=60,looseness=1.5] (3,0); \draw[->-=0] (3,0) to[in=300,out=270,looseness=1.5] (2,0);
\end{tikzpicture}
\caption{The base graph $B_1$.}
\label{fig_B_base}
\end{subfigure}
\begin{subfigure}[b]{.45\textwidth}\centering
\begin{tikzpicture}[scale=.9]
\draw[->-=.5] (0,0) node[circle,fill,inner sep=1.25]{} node[above]{$v_i$} -- (2,0) node[circle,fill,inner sep=1.25]{} node[above]{$v_t$};

\draw[-stealth] (2.45,0) -- (2.75,0);

\draw[->-=.5] (3.2,0) node[circle,fill,inner sep=1.25]{} node[above]{$v_i$} -- (4.2,0) node[circle,fill,inner sep=1.25]{};
\draw[->-=.5] (4.2,0) -- (5.2,0) node[circle,fill,inner sep=1.25]{} node[above]{$v_t$};
\draw (4.2,0) to[out=40,in=0,looseness=1.5] (4.2,0.8); \draw[->-=0] (4.2,0.8) to[out=180,in=140,looseness=1.5] (4.2,0);
\end{tikzpicture}
\caption{The replacement rule $e \to R$.}
\end{subfigure}
\caption{The Basilica replacement system.}
\label{fig_replacement_B}
\end{figure}

\begin{figure}\centering
\begin{subfigure}{.45\textwidth}\centering
\begin{tikzpicture}[scale=.52]
\draw[Green] (0,0) to[out=60,in=180] (1,0.75); \draw[Green] (1,0.75) to[out=0,in=120] (2,0);
\draw[Green] (2,0) to[out=240,in=0] (1,-0.75); \draw[Green] (1,-0.75) to[out=180,in=300] (0,0);
\draw[red] (0,0) to[out=240,in=270,looseness=1.5] (-1,0); \draw[red] (-1,0) to[out=90,in=120,looseness=1.5] (0,0) node[black,circle,fill,inner sep=1]{};
\draw[blue] (-1,0) to[out=120,in=90,looseness=1.5] (-1.5,0); \draw[blue] (-1.5,0) to[out=270,in=240,looseness=1.5] (-1,0) node[black,circle,fill,inner sep=1]{};
\draw[Plum] (2,0) to[out=60,in=90,looseness=1.5] (3,0); \draw[Plum] (3,0) to[out=270,in=300,looseness=1.5] (2,0) node[black,circle,fill,inner sep=1]{};
\draw[blue] (-1.3,0.15) node[above]{\footnotesize A};

\draw[-to] (3.325,0) -- (3.675,0);

\draw[red] (5,0) to[out=60,in=180] (6,0.75); \draw[red] (6,0.75) to[out=0,in=120] (7,0);
\draw[red] (7,0) to[out=240,in=0] (6,-0.75); \draw[red] (6,-0.75) to[out=180,in=300] (5,0);
\draw[blue] (5,0) to[out=240,in=270,looseness=1.5] (4,0); \draw[blue] (4,0) to[out=90,in=120,looseness=1.5] (5,0) node[black,circle,fill,inner sep=1]{};
\draw[Green] (7,0) to[in=90,out=60,looseness=1.5] (8,0); \draw[Green] (8,0) to[out=270,in=300,looseness=1.5] (7,0) node[black,circle,fill,inner sep=1]{};
\draw[Plum] (8,0) to[out=60,in=90,looseness=1.5] (8.5,0); \draw[Plum] (8.5,0) to[out=270,in=300,looseness=1.5] (8,0) node[black,circle,fill,inner sep=1]{};
\draw[blue] (4.4,0.35) node[above]{\footnotesize A};
\end{tikzpicture}
\end{subfigure}
\begin{subfigure}{.45\textwidth}\centering
\begin{tikzpicture}[scale=.52]
\draw[Magenta] (0,0) to[out=60,in=180] (1,0.75); \draw[Magenta] (1,0.75) to[out=0,in=120] (2,0);
\draw[Orange] (2,0) to[out=240,in=0] (1,-0.75); \draw[Green] (1,-0.75) to[out=180,in=300] (0,0);
\draw[red] (0,0) to[out=240,in=270,looseness=1.5] (-1,0); \draw[red] (-1,0) to[out=90,in=120,looseness=1.5] (0,0) node[black,circle,fill,inner sep=1]{};
\draw (2,0) to[out=60,in=90,looseness=1.5] (3,0); \draw (3,0) to[out=270,in=300,looseness=1.5] (2,0) node[circle,fill,inner sep=1]{};
\draw[blue] (1,-0.75) to[out=210,in=180,looseness=1.5] (1,-1.25); \draw[blue] (1,-1.25) to[out=0,in=330,looseness=1.5] (1,-0.75) node[black,circle,fill,inner sep=1]{};
\draw (2.6,0.35) node[above]{\footnotesize A};

\draw[-to] (3.325,0) -- (3.675,0);

\draw (5,0)[Green] to[out=60,in=180] (6,0.75); \draw[Magenta] (6,0.75) to[out=0,in=120] (7,0);
\draw[Orange] (7,0) to[out=240,in=0] (6,-0.75); \draw[Orange] (6,-0.75) to[out=180,in=300] (5,0);
\draw[blue] (5,0) to[out=240,in=270,looseness=1.5] (4,0); \draw[blue] (4,0) to[out=90,in=120,looseness=1.5] (5,0) node[black,circle,fill,inner sep=1]{};
\draw (7,0) to[in=90,out=60,looseness=1.5] (8,0); \draw (8,0) to[out=270,in=300,looseness=1.5] (7,0) node[circle,fill,inner sep=1]{};
\draw[red] (6,0.75) to[out=150,in=180,looseness=1.5] (6,1.25); \draw[red] (6,1.25) to[out=0,in=30,looseness=1.5] (6,0.75) node[black,circle,fill,inner sep=1]{};
\draw (7.6,0.35) node[above]{\footnotesize A};
\end{tikzpicture}
\end{subfigure}\\
\vspace*{4pt}
\begin{subfigure}{.45\textwidth}\centering
\begin{tikzpicture}[scale=.52]
\draw[Orange] (0,0) node[circle,fill,inner sep=1]{} to[out=60,in=225] (0.375,0.5);
\draw[Green] (0.375,0.5) to[out=45,in=180] (1,0.75);
\draw[blue] (0.375,0.5) to[out=75,in=45,looseness=1.5] (0.2,0.675); \draw[blue] (0.2,0.675) to[out=225,in=195,looseness=1.5] (0.375,0.5) node[black,circle,fill,inner sep=1]{};
\draw[Magenta] (1,0.75) to[out=0,in=135] (1.625,0.5);
\draw[Magenta] (1.625,0.5) to[out=315,in=120] (2,0);
\draw (2,0) to[out=240,in=0] (1,-0.75); \draw (1,-0.75) to[out=180,in=300] (0,0);
\draw (0,0) to[out=240,in=270,looseness=1.5] (-1,0); \draw (-1,0) to[out=90,in=120,looseness=1.5] (0,0) node[black,circle,fill,inner sep=1]{};
\draw (2,0) to[out=60,in=90,looseness=1.5] (3,0); \draw (3,0) to[out=270,in=300,looseness=1.5] (2,0) node[circle,fill,inner sep=1]{};
\draw[red] (1,0.75) to[out=150,in=180,looseness=1.5] (1,1.25); \draw[red] (1,1.25) to[out=0,in=30,looseness=1.5] (1,0.75) node[black,circle,fill,inner sep=1]{};
\draw (-.6,0.35) node[above]{\footnotesize A};

\draw[-to] (3.325,0) -- (3.675,0);

\draw[Orange] (5,0) node[black,circle,fill,inner sep=1]{} to[out=60,in=225] (5.375,0.5);
\draw[Orange] (5.375,0.5) to[out=45,in=180] (6,0.75);
\draw[Green] (6,0.75) to[out=0,in=135] (6.625,0.5);
\draw[Magenta] (6.625,0.5) to[out=315,in=120] (7,0);
\draw[red] (6.625,0.5) to[out=345,in=315,looseness=1.5] (6.8,0.675); \draw[red] (6.8,0.675) to[out=135,in=105,looseness=1.5] (6.625,0.5) node[black,circle,fill,inner sep=1]{};
\draw (7,0) to[out=240,in=0] (6,-0.75); \draw (6,-0.75) to[out=180,in=300] (5,0);
\draw (5,0) to[out=240,in=270,looseness=1.5] (4,0); \draw (4,0) to[out=90,in=120,looseness=1.5] (5,0) node[black,circle,fill,inner sep=1]{};
\draw (7,0) to[in=90,out=60,looseness=1.5] (8,0); \draw (8,0) to[out=270,in=300,looseness=1.5] (7,0) node[circle,fill,inner sep=1]{};
\draw[blue] (6,0.75) to[out=150,in=180,looseness=1.5] (6,1.25); \draw[blue] (6,1.25) to[out=0,in=30,looseness=1.5] (6,0.75) node[black,circle,fill,inner sep=1]{};
\draw (4.4,0.35) node[above]{\footnotesize A};
\end{tikzpicture}
\end{subfigure}
\begin{subfigure}{.45\textwidth}\centering
\begin{tikzpicture}[scale=.52]
\draw[Green] (0,0) to[out=60,in=180] (1,0.75); \draw[Green] (1,0.75) to[out=0,in=120] (2,0);
\draw[Magenta] (2,0) to[out=240,in=0] (1,-0.75); \draw[Magenta] (1,-0.75) to[out=180,in=300] (0,0);
\draw[red] (0,0) to[out=240,in=270,looseness=1.5] (-1,0); \draw[red] (-1,0) to[out=90,in=120,looseness=1.5] (0,0) node[black,circle,fill,inner sep=1]{};
\draw[blue] (2,0) to[out=60,in=90,looseness=1.5] (3,0); \draw[blue] (3,0) to[out=270,in=300,looseness=1.5] (2,0) node[black,circle,fill,inner sep=1]{};
\draw[red] (-.6,0.35) node[above]{\footnotesize A};

\draw[-to] (3.325,0) -- (3.675,0);

\draw[Magenta] (5,0) to[out=60,in=180] (6,0.75); \draw[Magenta] (6,0.75) to[out=0,in=120] (7,0);
\draw[Green] (7,0) to[out=240,in=0] (6,-0.75); \draw[Green] (6,-0.75) to[out=180,in=300] (5,0);
\draw[blue] (5,0) to[out=240,in=270,looseness=1.5] (4,0); \draw[blue] (4,0) to[out=90,in=120,looseness=1.5] (5,0) node[black,circle,fill,inner sep=1]{};
\draw[red] (7,0) to[in=90,out=60,looseness=1.5] (8,0); \draw[red] (8,0) to[out=270,in=300,looseness=1.5] (7,0) node[black,circle,fill,inner sep=1]{};
\draw[red] (7.6,0.35) node[above]{\footnotesize A};
\end{tikzpicture}
\end{subfigure}
\caption{The four generators of the Basilica rearrangement group $T_B$.}
\label{fig_TB_generators}
\end{figure}
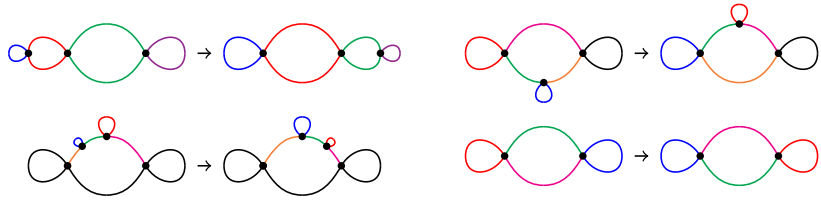

\section{The Airplane limit space and its rearrangements}\label{section_Airplane}

In this section we give a few definitions that are useful to work with the Airplane limit space, and in Subsection \ref{subsection_elements_TA} we exhibit five important elements of $T_A$ that will later turn out to generate the entire group.

\subsection{Components of the Airplane limit space}

Consider the Airplane replacement system depicted in Figure \ref{fig_replacement_A}. Note that each expansion of the base graph is a planar graph, thus it can be embedded in the Euclidean plane. When a blue edge is expanded, two new red edges are generated, and this pair of red edges encloses a connected region of the plane. These regions appear in each subsequent graph of the full expansion sequence, and then in the limit space, as exemplified in Figure \ref{fig_A_comp}. We call \textbf{component} any of these regions. The component colored in \textcolor{blue}{blue} in Figure \ref{fig_A_comp} is called \textbf{central component} and it is denoted by $C_0$.
It is not hard to see that elements of $T_A$ map components to components.

\begin{figure}\centering
\begin{overpic}[width=.65\textwidth]{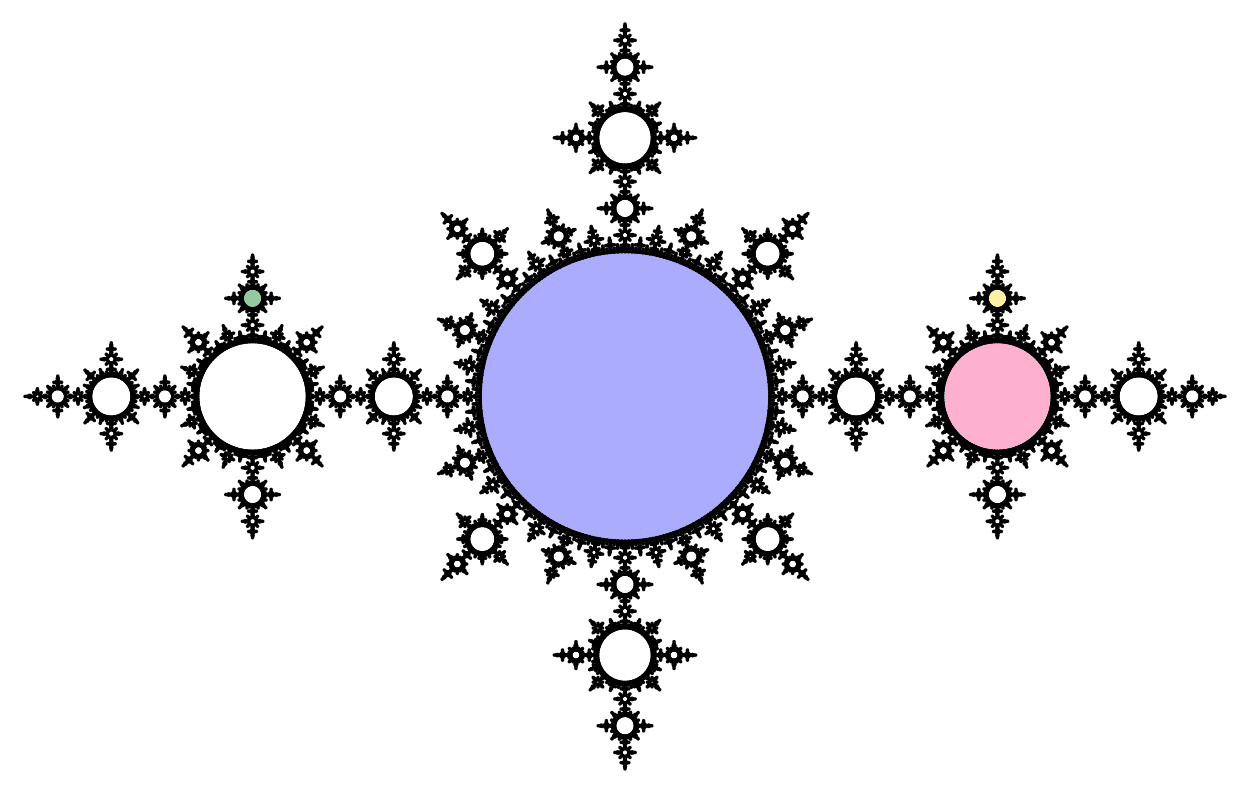}
    \put (48,31) {\textcolor{blue}{$C_0$}}
    \put (13.5,41.5) {\textcolor{ForestGreen}{$C_1$}}
    \put (82.5,22.5) {\textcolor{Magenta}{$C_2$}}
    \put (82,41.5) {\textcolor{YellowOrange}{$C_3$}}
\end{overpic}
\caption{Examples of components of the Airplane limit space.}
\label{fig_A_comp}
\end{figure}

Observe that the boundary $\partial C$ of any component $C$ is homeomorphic to $S^1 = [0,2\pi] / \{0,2\pi\}$, and there is a natural homeomorphism under which the dyadic angles of $S^1$ (i.e., elements of $2\pi \mathbb{Z}[\frac{1}{2}] / \{0,2\pi\} \subset S^1$) correspond to vertices between red edges. We will then refer to points of $\partial C$ as \textbf{angles} of $S^1$ under this homeomorphism, under counterclockwise orientation.

\subsection{Rays of the Airplane limit space}

We have already seen what blue cells look like in Figure \ref{fig_cells_A}. We say that a blue cell is a \textbf{primary blue cell} if it is maximal according to the set inclusion order. Note that these cells are adjacent to the central component, and that every other blue cell is included in exactly one of the primary blue cells.

Now, consider a blue edge $e$.
The corresponding cell must depart from a certain component $C$, and it identifies a segment departing perpendicularly at some dyadic angle of $\partial C$. When expanding $e$, the segment is divided in three parts: the central one is split into two red edges that make up the boundary of a component, and the two other parts result in new blue edges that lie on the segment. Expanding these blue edges, each half of the segment is broken again in a similar fashion. Expanding the red edges generate new blue edges departing from the associated component. Then the blue cell $\chi(e)$ is made up of the segment identified by $e$, after it has been split into the boundaries of the infinitely many components that lie on that segment, along with the infinitely many other blue cells departing from each of those components.

Consider the blue cell $\chi(e)$ and subtract from it its two extremes and all of the blue cells that it includes and that do not lie on the segment. We call \textbf{interval} the resulting set, and an example is given in Figure \ref{fig_A_ray}, denoted by $\textcolor{Green}{I}$. Note that intervals are quite similar to blue cells, but not the same: each interval corresponds to the segment associated to some blue edge $e$, after it has been split into the boundaries of the infinitely many components that lie on it, but it does not include anything else that departs from these components. In this sense, intervals follow a unique direction, while blue cells also expand from many points of the boundaries of the included components into new rays.

What we are really interested in are those intervals that are maximal with respect to the set inclusion order, which means that they represent an entire segment departing from a component. We call these \textbf{rays}, and if a ray departs from the central component we call it a \textbf{central ray}. Figure \ref{fig_A_ray} shows two examples of rays, \textcolor{blue}{$R_1$} and \textcolor{red}{$R_2$}, of which \textcolor{blue}{$R_1$} is a central ray. Note that, differently from intervals, each ray is uniquely determined by the component it departs from along with the angle it departs at. This idea will be explored further in Subsection \ref{subsection_component_paths}.

We also define the \textbf{right} and \textbf{left horizontal rays} as the central rays departing from the central component $C_0$ with angle $0$ and $\pi$, respectively, and we denote the right one by $R_0$. We then define $Hor$ (which stands for \textit{horizon}) as the union of the two horizontal rays and the boundary of the central component. $Hor$ and $C_0$ will have great importance in the study of $T_A$, as described in the next section.

Finally, given two components, we say that they are \textbf{related} if one lies on a ray departing from the other. For example, in Figure \ref{fig_A_comp} the component \textcolor{Magenta}{$C_2$} is related to both the components \textcolor{YellowOrange}{$C_3$} (small on the right) and \textcolor{blue}{$C_0$}, but these last two are not themselves related. Instead, the component \textcolor{ForestGreen}{$C_1$} (small on the left) is not related to any of the components colored in this picture.

\begin{figure}\centering
\begin{overpic}[width=.65\textwidth]{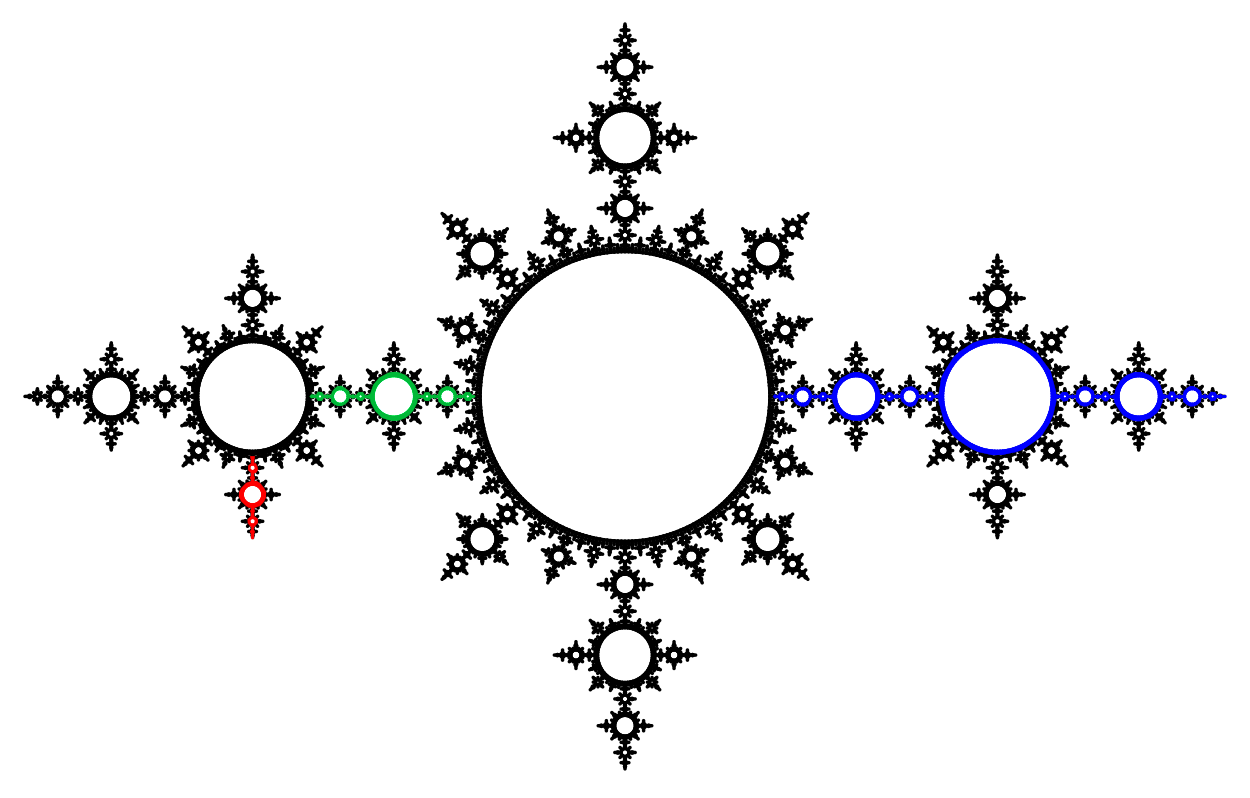}
    \put (84,22.5) {\textcolor{blue}{$R_1$}}
    \put (29,38.5) {\textcolor{Green}{$I$}}
    \put (13,20) {\textcolor{red}{$R_2$}}
\end{overpic}
\caption{Examples of an interval, highlighted in \textcolor{Green}{green} and two rays. The \textcolor{blue}{blue} ray is central, while the \textcolor{red}{red} one is not.}
\label{fig_A_ray}
\end{figure}

\subsection{\texorpdfstring{The elements of $T_A$}{The elements of TA}}\label{subsection_elements_TA}

Figure \ref{fig_generators} depict five specific elements of $T_A$: $\alpha, \beta, \gamma, \delta$ and $\varepsilon$. In this pictures, gray dotted lines show the action of these rearrangements. In Section \ref{section_generators_Airplane} we will show that these elements generate the entire group $T_A$.

\begin{figure}\centering
\begin{subfigure}[b]{.475\textwidth}\centering
\vspace*{10pt}
\begin{tikzpicture}[scale=.65]
    \draw[-latex,dotted,gray] (-1.05,0.2) to[out=45,in=150] (-0.4,0.4);
    \draw[-latex,dotted,gray] (0.4,0.4) to[out=30,in=120] (1.2,0.05);
    
    \draw[white] (0,-0.5) -- (0,-1);
    
    \draw[blue] (-2,0) -- node[midway,above]{\footnotesize A} (-1.5,0);
    \draw[red] (-1.25,0) circle (0.25);
    \draw[Green] (-1,0) -- (-0.5,0);
    \draw[Magenta] (0,0) circle (0.5);
    \draw[TealBlue] (0.5,0) -- (2,0);
    
    \draw[-to] (2.28,0) -- node[midway,above]{$\alpha$} (2.52,0);
    
    \draw[blue] (2.8,0) -- node[midway,above]{\footnotesize A} (4.3,0);
    \draw[red] (4.8,0) circle (0.5);
    \draw[Green] (5.3,0) -- (5.8,0);
    \draw[Magenta] (6.05,0) circle (0.25);
    \draw[TealBlue] (6.3,0) -- (6.8,0);
\end{tikzpicture}
\end{subfigure}
\begin{subfigure}[b]{.475\textwidth}\centering
\vspace*{10pt}
\begin{tikzpicture}[scale=.65]
    \draw[-latex,dotted,gray] (-0.05,-0.75) to[out=180,in=270] (-1.25,-0.05);
    \draw[-latex,dotted,gray] (-1.25,0.05) to[out=90,in=180] (-0.05,0.75);
    
    \draw[Magenta] (-2,0) -- (-0.5,0);
    \draw[Orange] (0.5,0) arc [radius=0.5, start angle=0, end angle=180];
    \draw[Green] (-0.5,0) arc [radius=0.5, start angle=180, end angle=270];
    \draw[blue] (0,-0.5) arc [radius=0.5, start angle=270, end angle=360];
    \draw (0.5,0) -- node[midway,above]{\footnotesize A} (2,0);
    \draw[red] (0,-0.5) -- (0,-1);
    
    \draw[-to] (2.28,0) -- node[midway,above]{$\beta$} (2.52,0);
    
    \draw[red] (2.8,0) -- (4.3,0);
    \draw[Orange] (5.3,0) arc [radius=0.5, start angle=0, end angle=90];
    \draw[Green] (4.8,0.5) arc [radius=0.5, start angle=90, end angle=180];
    \draw[blue] (4.3,0) arc [radius=0.5, start angle=180, end angle=360];
    \draw (5.3,0) -- node[midway,above]{\footnotesize A} (6.8,0);
    \draw[Magenta] (4.8,0.5) -- (4.8,1);
\end{tikzpicture}
\end{subfigure}\\
\begin{subfigure}[b]{.475\textwidth}\centering
\vspace*{10pt}
\begin{tikzpicture}[scale=.65]
    \draw[-latex,dotted,gray] (-0.53,0.53) arc [radius=0.75, start angle=135, end angle=95];
    \draw[-latex,dotted,gray] (0,0.75) arc [radius=0.75, start angle=90, end angle=45];
    
    \draw[-latex,dotted,white] (1,0) arc [radius=1, start angle=360, end angle=180];
    
    \draw (-2,0) -- node[midway,above]{\footnotesize A} (-0.5,0);
    \draw (-0.5,0) arc [radius=0.5, start angle=180, end angle=360];
    \draw (0.5,0) -- (2,0);
    \draw[Magenta] (0,0.5) -- (0,1);
    \draw[red] (-0.35,0.35) -- (-0.6,0.6);
    \draw[blue] (0.5,0) arc [radius=0.5, start angle=0, end angle=90];
    \draw[Green] (0,0.5) arc [radius=0.5, start angle=90, end angle=135];
    \draw[Plum] (-0.5,0) arc [radius=0.5, start angle=180, end angle=135];
    
    \draw[-to] (2.28,0) -- node[midway,above]{$\gamma$} (2.52,0);
    
    \draw (2.8,0) -- node[midway,above]{\footnotesize A} (4.3,0);
    \draw (4.3,0) arc [radius=0.5, start angle=180, end angle=360];
    \draw (5.3,0) -- (6.8,0);
    \draw[red] (4.8,0.5) -- (4.8,1);
    \draw[Magenta] (5.15,0.35) -- (5.4,0.6);
    \draw[blue] (5.3,0) arc [radius=0.5, start angle=0, end angle=45];
    \draw[Green] (4.8,0.5) arc [radius=0.5, start angle=90, end angle=45];
    \draw[Plum] (4.8,0.5) arc [radius=0.5, start angle=90, end angle=180];
\end{tikzpicture}
\end{subfigure}
\begin{subfigure}[b]{.475\textwidth}\centering
\vspace*{10pt}
\begin{tikzpicture}[scale=.65]
    \draw[-latex,dotted,gray] (-1,0) arc [radius=1, start angle=180, end angle=0];
    \draw[-latex,dotted,gray] (1,0) arc [radius=1, start angle=360, end angle=180];
    
    \draw[red] (-2,0) -- node[midway,above]{\footnotesize A} (-0.5,0);
    \draw[blue] (0.5,0) arc [radius=0.5, start angle=0, end angle=180];
    \draw[Green] (-0.5,0) arc [radius=0.5, start angle=180, end angle=360];
    \draw[Magenta] (0.5,0) -- (2,0);
    
    \draw[-to] (2.28,0) -- node[midway,above]{$\delta$} (2.52,0);
    
    \draw[Magenta] (2.8,0) -- (4.3,0);
    \draw[Green] (5.3,0) arc [radius=0.5, start angle=0, end angle=180];
    \draw[blue] (4.3,0) arc [radius=0.5, start angle=180, end angle=360];
    \draw[red] (5.3,0) -- node[midway,above]{\footnotesize A} (6.8,0);
\end{tikzpicture}
\end{subfigure}\\
\begin{subfigure}[b]{.475\textwidth}\centering
\vspace*{10pt}
\begin{tikzpicture}[scale=.65]
    \draw[-latex,dotted,gray] (0.75,0.1) to[out=75,in=105,looseness=1.625] (1.2,0.3);
    \draw[-latex,dotted,gray] (1.3,0.3) to[out=75,in=105,looseness=1.625] (1.75,0.05);
    
    \draw (-2,0) -- node[midway,above]{\footnotesize A} (-0.5,0);
    \draw (0,0) circle (0.5);
    \draw[Magenta] (0.5,0) -- (0.65,0);
    \draw[blue] (0.75,0) circle (0.1);
    \draw[Green] (0.85,0) -- (1,0);
    \draw[red] (1.25,0) circle (0.25);
    \draw[Plum] (1.5,0) -- (2,0);
    
    \draw[-to] (2.28,0) -- node[midway,above]{$\varepsilon$} (2.52,0);
    
    \draw (2.8,0) -- node[midway,above]{\footnotesize A} (4.3,0);
    \draw (4.8,0) circle (0.5);
    \draw[Magenta] (5.3,0) -- (5.8,0);
    \draw[blue] (6.05,0) circle (0.25);
    \draw[Green] (6.3,0) -- (6.45,0);
    \draw[red] (6.55,0) circle (0.1);
    \draw[Plum] (6.65,0) -- (6.8,0);
\end{tikzpicture}
\end{subfigure}
\caption{The five generators of $T_A$.}
\label{fig_generators}
\end{figure}
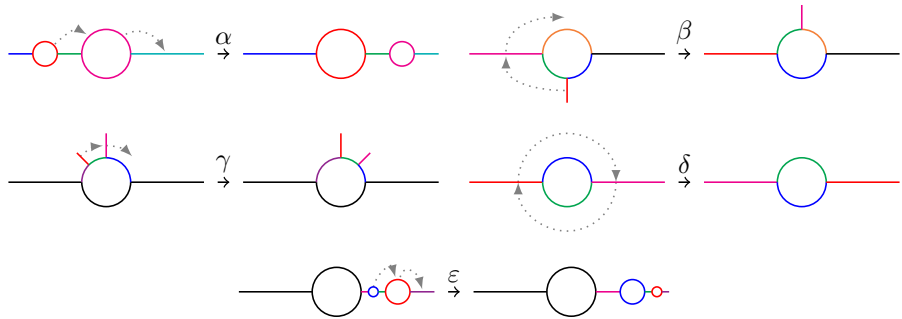

Those who are familiar with the Basilica rearrangement group $T_B$ studied in \cite{Belk_2015} might have noted that the first four of the five rearrangements we just introduced are very similar to the four rearrangements of the Basilica depicted in Figure \ref{fig_TB_generators}, which generate the entire $T_B$; this idea will be discussed in Section \ref{section_TB_in_TA}. The element $\varepsilon$ instead does not correspond to any element of $T_B$; we will see in Subsection \ref{subsection_F_in_TA} that it allows $T_A$ to act on rays as Thompson's group $F$ does on $(0,1)$.

\subsection{Component paths}\label{subsection_component_paths}

Consider the Airplane replacement system $\mathcal{A}$ (Figure \ref{fig_replacement_A}). We ignore the direction of the edges in the replacement system and give a new orientation as described below. Also, we identify each angle $2\pi k$ with the number $k$, which is the point of $S^1 = \frac{[0,1]}{\{0,1\}}$ that corresponds to the angle.

Note that, at each step of the full expansion sequence, rays are generated by halving red edges and components are generated by halving blue edges. This means that rays departing from a fixed component correspond to certain dyadic numbers, and the same holds for components lying on a fixed ray. In particular, we identify each ray with the unit interval $(0,1)$, where $0$ corresponds to the inner extreme, and we identify the border of each component with $S^1$, where the angle $0$ corresponds to the direction that needs to be taken to travel back towards the central component (or the right direction, if the component is the central one itself).

\begin{figure}[t]\centering
\includegraphics[width=.6\textwidth]{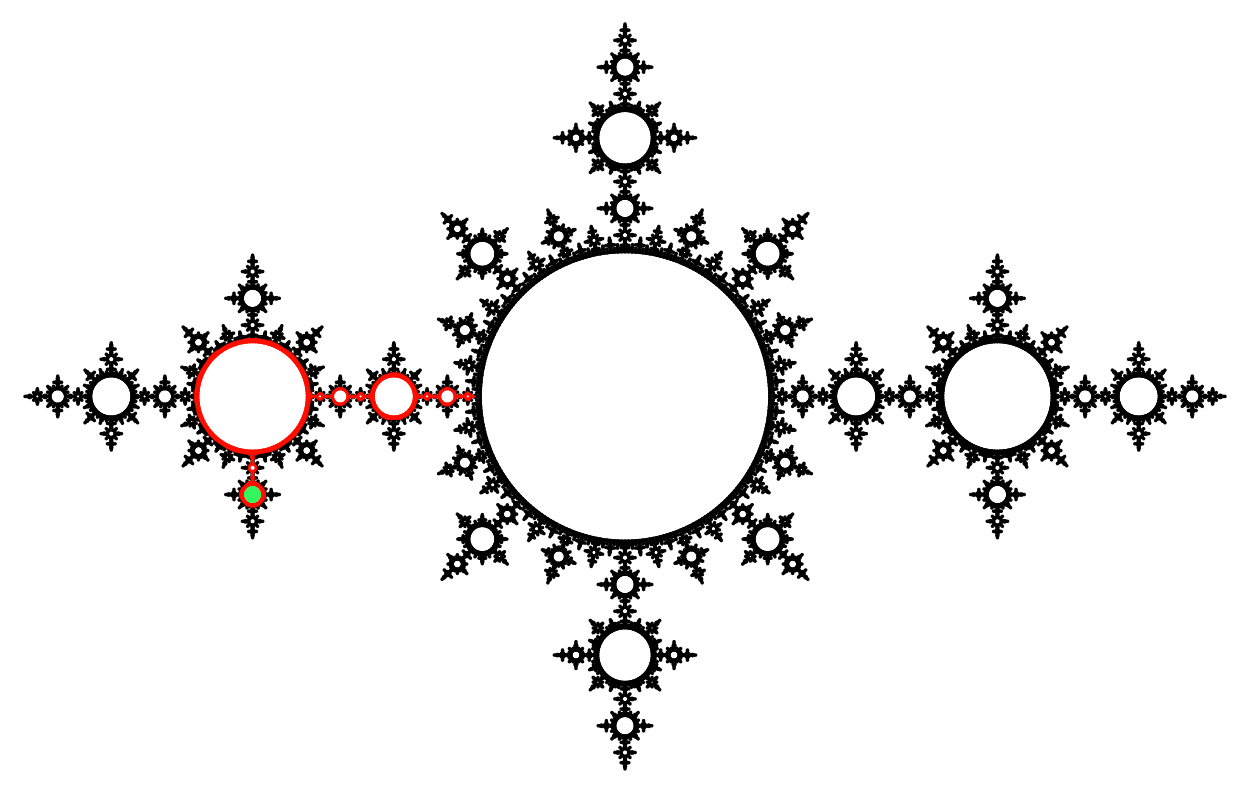}
\caption{The component path $\big( (\frac{1}{2},\frac{1}{2}), (\frac{3}{4},\frac{1}{2}) \big)$, highlighted in \textcolor{red}{red}, and the component that it identifies, highlighted in \textcolor{Green}{green}.}
\label{fig_component_path}
\end{figure}

Let $C$ be the component colored in green in Figure \ref{fig_component_path}. We can build a ``path'' of a finite amount of pairwise related components starting from the central one and ending at $C$, and it is unique if we assume that it is minimal. This path is depicted in red in Figure \ref{fig_component_path}, and it is specified by the following indications:
\begin{enumerate}
    \item departing from the central component at angle $\frac{1}{2}$, which identifies a ray;
    \item traveling on that ray for $\frac{1}{2}$ of its length, which identifies a component;
    \item departing from that component at angle $\frac{3}{4}$, which identifies a new ray;
    \item traveling on that ray for $\frac{1}{2}$ of its length, which identifies the component $C$.
\end{enumerate}
We then say that this path is identified by $\left( (\frac{1}{2}, \frac{1}{2}), (\frac{3}{4}, \frac{1}{2}) \right)$.

For each component $C$, there is a unique minimal path built like the one we just described, that goes from the central component $C_0$ to $C$. This path is identified by the list of components at which the path takes a turn, which is why we call such paths by the name \textbf{component paths}. Note that the empty path represents the central component.

When citing a component path, we will be referring to the component of $A$ that it identifies. We can represent the component path as the list $\left(C_1, C_2 \dots, C_n\right)$ of components at which the component path takes a turn. We can also use certain finite sequences of pairs of dyadic numbers: if $\big((\theta_1, l_1), (\theta_2, l_2), \dots , (\theta_n, l_n)\big)$ is such a sequence, then $\theta_1$ identifies the ray that departs from the central component at angle $2\pi \theta_1$, while $l_1$ identifies the component lying on the ray previously found that corresponds to the dyadic number $l_1$; the following pairs then work in the same way starting from the new component.

We also define the \textbf{depth} of a component in the Airplane limit space to be the number of non-central components in its component path, including itself unless it is the central component. Note that, if $\left(C_1, C_2 \dots, C_n\right)$ or $\big((\theta_1, l_1), (\theta_2, l_2), \dots , (\theta_n, l_n)\big)$ is a component path for the component $C=C_n$, then $n$ equals its depth.

\section{\texorpdfstring{Copies of Thompson's groups $F$ and $T$ in $T_A$}{Copies of Thompson's groups F and T in TA}}\label{section_thompson_in_TA}

In this section we exhibit two natural copies of Thompson's groups $T$ and $F$ contained in $T_A$ that will have great importance in studying the action of $T_A$.

We say that a rearrangement $f$ \textbf{extends canonically} on a cell $\chi(e)$ if it restricts to a canonical homeomorphism on $\chi(e)$, which intuitively means that $f$ maps $\chi(e)$ ``rigidly'' to some other cell of the same type.

\subsection{\texorpdfstring{A copy of Thompson's group $T$ in $T_A$}{A copy of Thompson's group T in TA}}\label{subsection_T_in_TA}

\begin{definition}
The \textbf{rigid stabilizer} of $C_0$, denoted by $rist(C_0)$, is the group of all elements of $T_A$ that map the component $C_0$ to itself and that extend canonically on the blue cells that depart from $C_0$.
\end{definition}
It is clear that this is a subgroup of $T_A$. The name “rigid” comes from the fact that an element of $rist(C_0)$ is determined entirely by its action on $\partial C_0$. Also note that we can equivalently define $rist(C_0)$ as the group of all elements of $T_A$ whose reduced graph pair diagrams have no component other that $C_0$.

\begin{theorem}\label{theorem_Airplane_rist_component}
$rist(C_0) = \langle \beta, \gamma, \delta \rangle \simeq T$, and it acts on $\partial C_0$ as $T$ does on $S^1$. In particular, its action is 2-transitive on the set of central rays.
\end{theorem}
The proof of this Theorem is essentially the same as the ones of Theorem 6.3 and Corollary 6.4 from \cite{Belk_2015}, with the proper definitions. It is important to note that $\beta, \gamma$ and $\delta$ (Figure \ref{fig_generators}) act on $\partial C_0$ exactly as the three generators $Y_0$, $Y_1$ and $Y_2$ of Thompson's group $T$ (Figure \ref{fig_T_gen}) act on $S^1$.

\subsection{\texorpdfstring{A copy of Thompson's group $F$ in $T_A$}{A copy of Thompson's group F in TA}}\label{subsection_F_in_TA}

\begin{definition}
The \textbf{rigid stabilizer} of $Hor$, denoted by $rist(Hor)$, is the group of all elements of $T_A$ that map the horizon $Hor$ to itself and that extend canonically on the red cells that make up components lying on $Hor$.
\end{definition}

It is clear that this is a subgroup of $T_A$. The name ``rigid'' comes from the fact that an element of $rist(Hor)$ is determined entirely by its action on $Hor$. Also note that we can equivalently define $rist(Hor)$ as the group of all elements of $T_A$ whose reduced graph pair diagrams have no ray other than the two horizontal ones.

\begin{theorem}\label{theorem_Airplane_rist_hor}
$rist(Hor) = \langle \alpha, \varepsilon \rangle \simeq F$, and it acts on $Hor$ as $F$ foes on $[0,1]$. In particular, its action is transitive on the set of components lying on the horizon.
\end{theorem}

The proof of this Theorem is very similar to the one of Theorem \ref{theorem_Airplane_rist_component}, itself similar to the ones of Theorem 6.3 and Corollary 6.4 from \cite{Belk_2015}. Note that the elements $\alpha$ and $\varepsilon$ (Figure \ref{fig_generators}) act on $Hor$ exactly as the two generators $X_0$ and $X_1$ of Thompson's group $F$ (Figure \ref{fig_F_gen}) act on $[0,1]$.

\medskip
In a similar manner we can study the subgroup $\langle \varepsilon, \varepsilon^{\alpha^{-1}} \rangle$ of $rist(Hor)$, which we denote by $rist(R_0)$. The two generators act on the right horizontal ray $R_0$ as the generators $X_0$ and $X_1$ of Thompson's group $F$ act on $[0,1]$, so we have:
\begin{proposition}\label{proposition_rist_ray}
$rist(R_0) \simeq F$, and it acts on $R_0$ as $F$ does on $[0,1]$. In particular, its action is transitive on the set of components lying on $R_0$.
\end{proposition}

We will make use of this subgroup $rist(R_0)$ in the proofs of Lemma \ref{lemma_TA_transitive} and Theorem \ref{theorem_commutator_TA_finitely_generated}.

\section{\texorpdfstring{Generators of $T_A$}{Generators of TA}}\label{section_generators_Airplane}

In this section we prove that $\alpha, \beta, \gamma, \delta$ and $\varepsilon$ (depicted in Figure \ref{fig_generators}) generate $T_A$.

\begin{lemma}\label{lemma_TA_transitive}
$\langle \alpha, \beta, \gamma, \delta, \varepsilon \rangle$ acts transitively on the set of components of $A$.
\end{lemma}

\begin{proof}
Let $C_n$ be any component of depth $n$. We will show that $C_n$ can be mapped to the central component by an element of $\langle \alpha, \beta, \gamma, \delta, \varepsilon \rangle$ by induction on $n$.

If $n = 0$ then $C_n$ is itself the central component and we are done. Otherwise, $C_n$ is connected to the central component $C_0$ by a component path $(C_0,\, C_1,\dots,\, C_n)$. Consider the component $C_1$, which is related to the central component, thus it lies on some central ray $R$. Because of Theorem \ref{theorem_Airplane_rist_component}, there exists an element $f \in \langle \beta, \gamma, \delta \rangle$ for which $f(R)$ is the right horizontal ray. Then $f(C_1)$ must be a component lying on the right horizontal ray and so, because of Proposition \ref{proposition_rist_ray}, there is a $g \in \langle \varepsilon, \varepsilon^{\alpha^{-1}} \rangle$ such that $(g \circ f) (C_1)$ is the component in the middle of the right horizontal ray. Then clearly $(\alpha^{-1} \circ g \circ f) (C_1)$ is the central component $C_0$.

It is easy to see that $((\alpha^{-1} \circ g \circ f) (C_1), \dots , \, (\alpha^{-1} \circ g \circ f) (C_n))$ is a component path, hence $(\alpha^{-1} \circ g \circ f) (C_n)$ has depth $n-1$. By our induction hypothesis, $(\alpha^{-1} \circ g \circ f) (C_n)$ can be mapped to $C_0$ by $\langle \alpha, \beta, \gamma, \delta, \varepsilon \rangle$, thus $C_n$ can as well.
\end{proof}

We will now prove that these five elements generate the entire $T_A$. Recall that, if $f$ and $g$ are rearrangements, in order to compute their composition $f \circ g$ we need to expand both their graph pair diagrams so that the range graph for $g$ is the same as the domain graph for $f$.

\begin{theorem}\label{theorem_TA_generators}
The group $T_A$ is generated by the elements $\{\alpha, \beta, \gamma, \delta, \varepsilon \}$.
\end{theorem}

\begin{proof}
Let $f \in T_A$. By the previous lemma, up to replacing $f$ by $g \circ f$ for a suitable $g \in \langle \alpha, \beta, \gamma, \delta, \varepsilon \rangle$, we may assume that $f(C_0) = C_0$. We must show that $f \in \langle \alpha, \beta, \gamma, \delta, \varepsilon \rangle$. We proceed by induction on the number $n$ of non-central components in the reduced graph pair diagram for $f$.

The base graph for the Airplane does not contain any non-central component, so the base case is $n=0$. If $n=0$, then $f \in rist(C_0)$, which is generated by $\beta, \gamma$ and $\delta$ because of Theorem \ref{theorem_Airplane_rist_component}, therefore $f$ belongs to $\langle \alpha, \beta, \gamma, \delta, \varepsilon \rangle$.

Suppose that $n \geq 1$. Since $f(C_0) = C_0$, the action of $f$ permutes the points of departure of central rays from $C_0$, which correspond precisely to dyadic points of $S^1$. Because of Theorem \ref{theorem_Airplane_rist_component}, there exists an element $g \in \langle \beta, \gamma, \delta \rangle$ that permutes the points of departure of the central rays in the same way as $f$. Then the composition $h := g^{-1} \circ f$ fixes each of these points. It suffices to prove that $h \in \langle \alpha, \beta, \gamma, \delta, \varepsilon \rangle$.

Since $g^{-1}$ does not have any non-central component in the domain graph for its reduced graph pair diagram, $h$ has exactly $n$ non-central components in its reduced graph pair diagram. Now, call $p_1, \dots,\, p_m$ the points of departure of the central rays from $C_0$ in the reduced graph pair diagram for $h$. The right horizontal ray $R_0$ appears in every expansion of $\mathcal{A}$, so we can assume that $p_1$ is the point of adjacency between $C_0$ and $R_0$. Note that the component path for each non-central component travels through one and only one of these $p_i$. We now distinguish two cases:

\smallskip
\textit{\textbf{Case 1}:} Suppose that the reduced graph pair diagram for $h$ has non-central components with component paths that travel through more than one $p_i$ (as in Figure \ref{fig_A_diagram1}). Then we can express $h$ as a composition $h_1 \circ  \cdots \circ h_m$, where each $h_i$
has a graph pair diagram obtained from the reduced graph pair diagram for $h$ by removing all non-central components except for those with component paths that travel through $p_i$ (along with every ray that departs from the components removed this way). For example, Figure \ref{fig_A_diagram2} depicts $h_1$ and $h_2$ for the $h$ in Figure \ref{fig_A_diagram1}. Then each $h_i$ must have fewer than $n$ non-central components. By induction, it follows that each $h_i \in \langle \alpha, \beta, \gamma, \delta, \varepsilon \rangle$, and therefore $h \in \langle \alpha, \beta, \gamma, \delta, \varepsilon \rangle$ as well.

\begin{figure}[b]\centering
\begin{tikzpicture}[scale=1]
\draw (-2,0) -- (-0.5,0);
\draw (0,0) circle (0.5);
\draw (0.5,0) -- (0.65,0);
\draw (0.75,0) circle (0.1);
\draw (0.85,0) -- (1,0);
\draw (1.25,0) circle (0.25);
\draw (1.5,0) -- (2,0);

\draw (0,0.5) -- (0,1);
\draw (-0.225,0.75) -- (-0.1,0.75);
\draw[fill=white] (0,0.75) circle (0.1);

\node[blue] at (0.5,0) [circle,fill,inner sep=1.25]{};
\node[red] at (0,0.5) [circle,fill,inner sep=1.25]{};

\draw[-to] (2.28,0) -- node[midway,above]{$h$} (2.52,0);

\draw (2.8,0) -- (4.3,0);
\draw (4.8,0) circle (0.5);
\draw (5.3,0) -- (5.8,0);
\draw (6.05,0) circle (0.25);
\draw (6.3,0) -- (6.45,0);
\draw (6.55,0) circle (0.1);
\draw (6.65,0) -- (6.8,0);
\draw (4.8,0.5) -- (4.8,1);

\draw (4.8,0.5) -- (4.8,1);
\draw (5.025,0.75) -- (4.9,0.75);
\draw[fill=white] (4.8,0.75) circle (0.1);

\node[blue] at (5.3,0) [circle,fill,inner sep=1.25]{};
\node[red] at (4.8,0.5) [circle,fill,inner sep=1.25]{};
\end{tikzpicture}
\caption{An example for \textit{Case 1} of the proof of Theorem \ref{theorem_TA_generators}. The two colored points represent $p_1$ and $p_2$.
}
\label{fig_A_diagram1}
\end{figure}
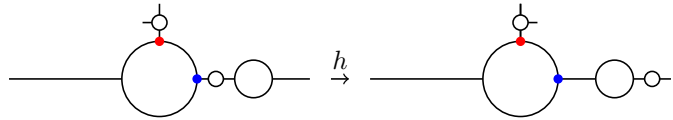

\begin{figure}[t]\centering
\vspace{10pt}
\begin{tikzpicture}[scale=1]
\draw (-2,0) -- (-0.5,0);
\draw (0,0) circle (0.5);
\draw (0.5,0) -- (0.65,0);
\draw (0.75,0) circle (0.1);
\draw (0.85,0) -- (1,0);
\draw (1.25,0) circle (0.25);
\draw (1.5,0) -- (2,0);
\node[blue] at (0.5,0) [circle,fill,inner sep=1.25]{};

\draw[-to] (2.28,0) -- node[midway,above]{$h_1$} (2.52,0);

\draw (2.8,0) -- (4.3,0);
\draw (4.8,0) circle (0.5);
\draw (5.3,0) -- (5.8,0);
\draw (6.05,0) circle (0.25);
\draw (6.3,0) -- (6.45,0);
\draw (6.55,0) circle (0.1);
\draw (6.65,0) -- (6.8,0);
\node[blue] at (5.3,0) [circle,fill,inner sep=1.25]{};
\end{tikzpicture}
\\
\vspace{16pt}
\begin{tikzpicture}[scale=1]
\draw (-2,0) -- (-0.5,0);
\draw (0,0) circle (0.5);
\draw (0.5,0) -- (2,0);

\draw (0,0.5) -- (0,1);
\draw (-0.225,0.75) -- (-0.1,0.75);
\draw[fill=white] (0,0.75) circle (0.1);
\node[red] at (0,0.5) [circle,fill,inner sep=1.25]{};

\draw[-to] (2.28,0) -- node[midway,above]{$h_2$} (2.52,0);

\draw (2.8,0) -- (4.3,0);
\draw (4.8,0) circle (0.5);
\draw (5.3,0) -- (6.8,0);
\draw (4.8,0.5) -- (4.8,1);

\draw (4.8,0.5) -- (4.8,1);
\draw (5.025,0.75) -- (4.9,0.75);
\draw[fill=white] (4.8,0.75) circle (0.1);
\node[red] at (4.8,0.5) [circle,fill,inner sep=1.25]{};
\end{tikzpicture}
\caption{The element $h$ in Figure \ref{fig_A_diagram1} can be written as
$h_1 \circ h_2$, since it has two center-adjacent components.}
\label{fig_A_diagram2}
\end{figure}
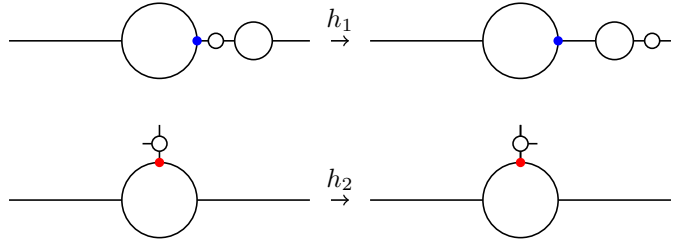

\smallskip
\textit{\textbf{Case 2}:} Suppose that all non-central components in the reduced graph pair diagram for $h$ have component paths that travel through the same $p_i$. Because of Theorem \ref{theorem_Airplane_rist_component}, there exists a $\sigma \in \langle \beta, \gamma, \delta \rangle$ such that $\sigma(p_1) = p_i$. Let $k := h^{\sigma}$, and note that we can conjugate $h$ by $\sigma$ without expanding the blue cells involved in the action of $h$, so $k$ has at most $n$ non-central components. Moreover, all of the non-central components of the reduced graph pair diagram for $k$ have component paths that travel through $p_1$. It suffices to prove that $k \in \langle \alpha, \beta, \gamma, \delta, \varepsilon \rangle$.

Consider the reduced graph pair diagram for $k$, which is sketched in Figure \ref{fig_A_diagram3}. Both the domain and the range graphs only have non-central components in the right primary blue cell. Let $C_D$ be the component on the right horizontal ray that is closest to $C_0$ in the domain graph and let $C_R$ be the component in the range graph that is the closest to $C_0$, which must be $k(C_D)$. Note that $C_D = \big( (0,\frac{1}{2^d}) \big)$ and $C_R = \big( (0,\frac{1}{2^r}) \big)$ for some $d\geq1$ and $r\geq1$. Therefore $\varepsilon^{d-1} (C_D) = \big( (0,1/2) \big)$ and $\varepsilon^{r-1} (C_R) = \big( (0,1/2) \big)$, so both the domain and the range graphs of the reduced graph pair diagram for $l := \varepsilon^{r-1} \circ k \circ \varepsilon^{-d+1}$ have $\big( (0,1/2) \big)$ as the closest component to $C_0$ on the right horizontal ray. Therefore both domain and range graphs in the reduced graph pair diagram for $l$ are such that no component lies on the inner half of $R_0$, as shown in Figure \ref{fig_A_diagram3}, and the number of non-central components in the reduced graph pair diagram for $l$ is the same as the one for $k$. It suffices to prove that $l \in \langle \alpha, \beta, \gamma, \delta, \varepsilon \rangle$.

We can now conjugate $l$ by $\alpha$ without performing any expansion of the cells involved in the action of $l$. The resulting graph pair diagram for $l^{\alpha} = \alpha^{-1} \circ l \circ \alpha$ is shown in \ref{fig_A_diagram3}: it is similar to the one of $l$, but the action is shifted left, and the central component in both the domain and range graphs for $l$ ends up in $\big( (1/2, 1/2) \big)$ in the graphs for $l^{\alpha}$, with no ray departing from it, so it can then be reduced. Then $l^{\alpha}$ has $n-1$ or less non-central components, so by induction we have that $l^{\alpha}$ belongs to $\langle \alpha, \beta, \gamma, \delta, \varepsilon \rangle$. Thus $l$ does too.

\begin{figure}\centering
\begin{subfigure}{\textwidth}\centering
\begin{tikzpicture}[scale=1]
\draw (-2,0) -- (-0.5,0);
\draw (0,0) circle (0.5);
\draw[blue] (0.5,0) -- (2,0);
\draw[red,fill=white] (0.75,0) circle (0.1); \node[red] at (0.75,0.3){$C_D$};
\draw[dashed,red] (0.9,0.05) -- (1.95,0.05);

\draw[-to] (2.28,0) -- node[midway,above]{$k$} (2.52,0);

\draw (2.8,0) -- (4.3,0);
\draw (4.8,0) circle (0.5);
\draw[blue] (5.3,0) -- (6.8,0);
\draw[red,fill=white] (5.75,0) circle (0.1); \node[red] at (5.75,0.3){$C_R$};
\draw[dashed,red] (5.9,0.05) -- (6.75,0.05);
\end{tikzpicture}
\end{subfigure}
\\
\vspace{16pt}
\begin{subfigure}{\textwidth}\centering
\begin{tikzpicture}[scale=1]
\draw (-2,0) -- (-0.5,0);
\draw (0,0) circle (0.5);
\draw (0.5,0) -- (1,0);
\draw[blue] (1.5,0) -- (2,0);
\draw[red,fill=white] (1.25,0) circle (0.25); \node[red] at (1.25,0.5){$C_D$};
\draw[dashed,red] (1.55,0.05) -- (1.95,0.05);

\draw[-to] (2.28,0) -- node[midway,above]{$l$} (2.52,0);

\draw (2.8,0) -- (4.3,0);
\draw (4.8,0) circle (0.5);
\draw (5.3,0) -- (5.8,0);
\draw[blue] (6.3,0) -- (6.8,0);
\draw[red,fill=white] (6.05,0) circle (0.25); \node[red] at (6.05,0.5){$C_R$};
\draw[dashed,red] (6.35,0.05) -- (6.75,0.05);
\end{tikzpicture}
\end{subfigure}
\\
\vspace{16pt}
\begin{subfigure}{\textwidth}\centering
\begin{tikzpicture}[scale=1]
\draw (-2,0) -- (-0.5,0);
\draw[fill=white] (-1.25,0) circle (0.25);
\draw[red,fill=white] (0,0) node{$C_D$} circle (0.5);
\draw[blue] (0.5,0) -- (2,0);
\draw[dashed,red] (0.55,0.05) -- (1.95,0.05);

\draw[-to] (2.28,0) -- node[midway,above]{$l^{\alpha}$} (2.52,0);

\draw (2.8,0) -- (4.3,0);
\draw[fill=white] (3.55,0) circle (0.25);
\draw[red,fill=white] (4.8,0) node{$C_R$} circle (0.5);;
\draw[blue] (5.3,0) -- (6.8,0);
\draw[dashed,red] (5.35,0.05) -- (6.75,0.05);

\node[white] at (1.25,0.5){$C_D$};
\end{tikzpicture}
\end{subfigure}
\caption{Sketches of $k, l$ and $l^{\alpha}$ for \textit{Case 2} of the proof of Theorem \ref{theorem_TA_generators}. Everything drawn in black is fixed pointwise, while everything colored may not be; rays may depart from red components, and components may lie on dashed red lines.}
\label{fig_A_diagram3}
\end{figure}
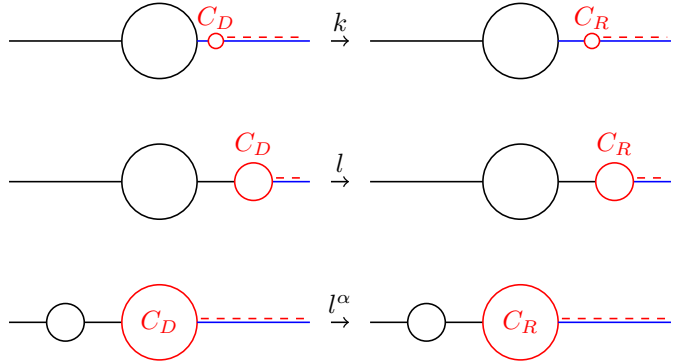
\end{proof}

\begin{remark}\label{remark_TA_generated_T_F}
Recall from Theorems \ref{theorem_Airplane_rist_component} and \ref{theorem_Airplane_rist_hor} that $rist(C_0) = \langle \beta, \gamma, \delta \rangle$ and $rist(Hor) = \langle \alpha, \varepsilon \rangle.$ Therefore, as a consequence of Theorem \ref{theorem_TA_generators}, we have that
\[ T_A = \langle rist(C_0), rist(Hor) \rangle, \]
where $rist(C_0)$ and $rist(Hor)$ are arguably the two ``most natural'' copies of Thompson's groups $T$ and $F$ in $T_A$, respectively.
\end{remark}

\begin{question}
We have just proved that $T_A$ is finitely generated. Is $T_A$ also finitely presented?
We recall that $T_B$ is not, as proved by Witzel and Zaremsky in \cite{witzar}.
\end{question}

\section{\texorpdfstring{The commutator subgroup of $T_A$}{The commutator subgroup of TA}}\label{section_commutator}

In this section we study the commutator subgroup of $T_A$. We first find a characterization of the rearrangements of $[T_A, T_A]$ in terms of their action on the extremities of the Airplane, and we find an infinite generating set along the way. Then we prove that $[T_A, T_A]$ is simple and finally we find a finite generating set for $[T_A, T_A]$.

\begin{remark}\label{remark_commutator_TA_contains}
It is easy to see that $[T_A, T_A]$ is quite large. Indeed:
\begin{enumerate}
    \item $[T_A, T_A] \geq rist(C_0)$, since $rist(C_0) \simeq T$ (Theorem \ref{theorem_Airplane_rist_component}) and $T = [T,T]$.
    \item $[T_A, T_A] \geq [rist(Hor),rist(Hor)]$, which is the group of those rearrangements of $rist(Hor)$ that act trivially at the left and right extremes of $Hor$, since $rist(Hor) \simeq F$ (Theorem \ref{theorem_Airplane_rist_hor}).
\end{enumerate}
Since $rist(C_0) = \langle \beta, \gamma, \delta \rangle$, we have that $\beta, \gamma, \delta \in [T_A, T_A]$. A direct
computation shows that $\alpha = [\varepsilon, \delta] \circ [\varepsilon^{-1}, \alpha^{-2}]$, so the commutator subgroup of $T_A$ also contains $\alpha$. Hence, we already know that $[T_A, T_A]$ contains four out of five generators of $T_A$. We will see along the way that it does not contain $\varepsilon$.
\end{remark}

\subsection{A characterization of the commutator subgroup}\label{subsection_extremal}

Let $\mathcal{E}$ be the set of all external extremes of rays in the Airplane limit space. Note that all rearrangements $f \in T_A$ act on the set $\mathcal{E}$ by permutation. We now define a concept of derivative of a rearrangement of $A$ around an extreme which gives an idea of how much the rearrangement dilates or shrinks the extremities of $A$ compared to the length of the respective ray. We then study the product of all of these derivatives, which gives a characterization for the commutator subgroup of $T_A$.

\phantomsection\label{length}
Let $e$ be a blue edge of an expansion of the Airplane replacement system. Note that $e$ lies on a unique ray $R$, and it corresponds to a portion of $R$: it can be half of it, a quarter, an eighth, etc. More precisely, it can be $1/2^k$ for any $k \in \mathbb{N}$. We call \textbf{length} of $e$ this number $1/2^k$. Figure \ref{fig_edge_length} depicts a few examples. Additionally, if $\chi(e)$ is the blue cell generated by the blue edge $e$, we define the \textbf{length} of $\chi(e)$ to be the same as the length of $e$.

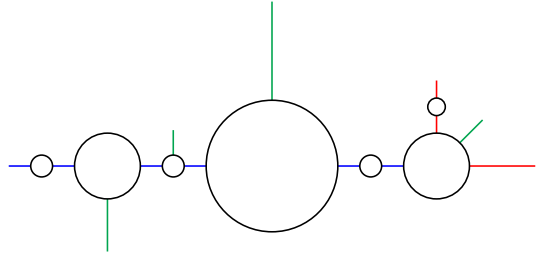
\begin{figure}[b]\centering
\begin{tikzpicture}[scale=1.8]
\draw[blue] (-2,0) -- node[midway,above]{$\frac{1}{4}$} (-1.833,0) -- (-1.667,0) -- node[midway,above]{$\frac{1}{4}$} (-1.5,0) -- (-1,0) -- node[midway,below]{$\frac{1}{4}$} (-.833,0) -- (-.667,0) -- node[midway,below]{$\frac{1}{4}$} (-.5,0) -- (.5,0) -- node[midway,below]{$\frac{1}{4}$} (.667,0) -- (.833,0) -- node[midway,below]{$\frac{1}{4}$} (1,0);
\draw[red] (1.5,0) -- node[midway,below]{$\frac{1}{2}$} (2,0);
\draw[red] (1.25,0.65) -- node[midway,left]{\footnotesize$1/2$} (1.25,.5167) -- (1.25,.3833) -- node[midway,left]{\footnotesize$1/2$} (1.25,.25);
\draw[Green] (-0.75,.083) -- node[midway,left]{\footnotesize$1$} (-0.75,0.273);
\draw[Green] (0,.5) -- node[midway,left]{\footnotesize$1$} (0,1.25);
\draw[Green] (-1.25,-.25) -- node[midway,left]{\footnotesize$1$} (-1.25,-0.65);
\draw[Green] (1.25,0) -- (1.6,.35) node[right]{\footnotesize$1$};

\draw[fill=white] (0,0) circle (0.5);
\draw[fill=white] (-1.25,0) circle (0.25);
\draw[fill=white] (1.25,0) circle (0.25);
\draw[fill=white] (0.75,0) circle (0.083);
\draw[fill=white] (-0.75,0) circle (0.083);
\draw[fill=white] (-1.75,0) circle (0.083);
\draw[fill=white] (1.25,0.45) circle (0.0666);
\end{tikzpicture}
\caption{Examples of lengths of blue edges: \textcolor{Green}{green}, \textcolor{red}{red} and \textcolor{blue}{blue} ones have lengths \textcolor{Green}{$1$}, \textcolor{red}{$\frac{1}{2}$} and \textcolor{blue}{$\frac{1}{4}$}, respectively.}
\label{fig_edge_length}
\end{figure}

Let $p \in \mathcal{E}$ and $f \in T_A$. Consider the reduced graph pair diagram for $f$. If $p$ does not appear in the domain graph, we define $D_p (f)$ as $1$. Otherwise, $p$ corresponds to the extremity of some blue edge $e_p$ in the graph. Then $f(e_p)$ is the blue edge that appears in the range graph whose extreme is $f(p)$. Let $1/2^d$ be the length of $e_p$ and $1/2^r$ the length of $f(e_p)$. We define the \textbf{extremal derivative} of $f$ in $p$ as the ratio of the length of $f(e_p)$ to the length of $e_p$, which is
\[ D_p (f) := 2^{r-d}. \]
This represents how much the action of $f$ dilates or shrinks the extremity around $p$ when compared to the ray of which $p$ and $f(p)$ are the external extremes.

Finally, we define the \textbf{global extremal derivative} of an element $f$ of $T_A$ as the product of all its extremal derivatives, which is
\[ D(f) := \prod\limits_{p \in \mathcal{E}} D_p (f). \]
Note that $D_p (f)$ equals $1$ for each extreme $p$ that does not appear in the domain graph of the reduced graph pair diagram for $f$, so this product only has a finite amount of non-trivial factors.

Now, note that $D_p (f \circ g) = D_{g(p)} (f) \cdot D_p (g)$, for all $f, g \in T_A$ and for all $p \in \mathcal{E}$. Therefore, since $g$ permutes the set $\mathcal{E}$, we have that
\[ D(f \circ g) = D(f) \cdot D(g). \]
Hence, the map $D: T_A \to \langle 2 \rangle_{ \mathbb{Q}^*}$ is group morphism, where $\langle 2 \rangle_{\mathbb{Q}^*}$ is the multiplicative group of all the integer powers of $2$, which is an infinite cyclic group.

\begin{theorem}
$[T_A,T_A] = Ker(D)$.
\end{theorem}

\begin{proof}
Note that $D (\varepsilon^{-k}) = 2^k$ for each integer $k$, so $D$ is surjective. Hence, $T_A / Ker(D) \simeq \langle 2 \rangle_{\mathbb{Q}^*}$, which is abelian, and so $[T_A, T_A] \leq Ker(D)$. We only need to prove that $Ker(D) \leq [T_A, T_A]$.

Let $f \in Ker(D)$. Since $T_A = \langle \alpha, \beta, \gamma, \delta, \varepsilon \rangle$, the rearrangement $f$ can be written as $f_1 \circ f_2 \dots f_k$, where each $f_i$ is a generator or an inverse of one. Since $D(f)=1$ and $D(\alpha) = D(\beta) = D(\gamma) = D(\delta) = 1$, the number of $\varepsilon$'s among the $f_i$ must be equal to the number of $\varepsilon^{-1}$'s. Then $f$ is the product of elements chosen in the set $\{ \alpha^{\varepsilon^j}, \beta^{\varepsilon^j}, \gamma^{\varepsilon^j}, \delta^{\varepsilon^j} \: | \: j \in \mathbb{Z} \}$. Since $\alpha, \beta, \gamma$ and $\delta$ all belong to the commutator subgroup (Remark \ref{remark_commutator_TA_contains}), and $[T_A, T_A]$ is normal in $T_A$, these elements all belong to $[T_A, T_A]$. Therefore, $f \in [T_A, T_A]$, and so $Ker(D) = [T_A, T_A]$.
\end{proof}

As a consequence of this, we can immediately find an infinite generating set:

\begin{corollary}
$[T_A, T_A] = \langle \beta, \gamma, \alpha^{\varepsilon^k}, \delta^{\varepsilon^k} \: | \: k \in \mathbb{Z} \rangle$, which is the normal closure of the subgroup $H := \langle \alpha, \beta, \gamma, \delta \rangle$ in $T_A$.
\label{corollary_commutator_TA_generators}
\end{corollary}

\begin{proof}
As seen in the proof of the previous theorem, $Ker(D) = \langle \alpha^{\varepsilon^j}, \beta^{\varepsilon^j}, \gamma^{\varepsilon^j}, \delta^{\varepsilon^j} \: | \: j \in \mathbb{Z} \rangle$, which is clearly the normal closure of $H$.

Since the supports of both $\beta$ and $\gamma$ have trivial intersection with the support of $\varepsilon$, we have that $\beta^{\varepsilon} = \beta$ and $\gamma^{\varepsilon} = \gamma$, so $[T_A, T_A] = Ker(D) = \langle \beta, \gamma, \alpha^{\varepsilon^k}, \delta^{\varepsilon^k} \: | \: k \in \mathbb{Z} \rangle$.
\end{proof}

Since $D: T_A \to \langle 2 \rangle_{\mathbb{Q}^*}$ is surjective and $\langle 2 \rangle_{\mathbb{Q}^*}$ is an infinite cyclic group, applying the first isomorphism theorem to $D$ immediately gives us the following result:

\begin{corollary}\label{corollary_abelianization}
The abelianization of $T_A$ is an infinite cyclic group. In particular, the index of $[T_A, T_A]$ in $T_A$ is infinite.
\end{corollary}

Finally, note that $[T_A, T_A] \cap \langle \varepsilon \rangle = \emptyset$ because $D(\varepsilon^k) = 1 \iff k=0$ while $[T_A, T_A] = Ker(D)$, and that $[T_A, T_A] \langle \varepsilon \rangle$ contains the five generators of $T_A$. Then:

\begin{corollary}
$T_A = [T_A, T_A] \rtimes \langle \varepsilon \rangle$.
\end{corollary}

\subsection{The simplicity of the commutator subgroup}\label{subsection_commutator_TA_simple}

We start by noting certain transitivity properties of $[T_A, T_A]$ that we will use later.

\begin{lemma}\label{lemma_commutator_TA_transitive}
$[T_A,T_A]$ acts transitively on the set of components of $A$.
\end{lemma}

The proof of this follows the same outline of the proof of Lemma \ref{lemma_TA_transitive}. By induction on the depth $n$ of a component $C_n$ whose component path is $(C_0, C_1, \dots C_n)$, we find $f \in \langle \beta, \gamma, \delta \rangle \leq [T_A, T_A]$ such that $f(C_1)$ lies on the right horizontal ray $R_0$. Then, since $rist(Hor) \simeq F$ and $[F,F]$ is transitive on the set of dyadic points of $(0,1)$, we find $g \in [rist(Hor), rist(Hor)] \leq [T_A, T_A]$ such that $g \circ f (C_1) = C_0$, and our induction hypothesis does the rest.

As a consequence of this Lemma, we have that:

\begin{corollary}\label{corollary_commutator_TA_transitive_points}
The commutator subgroup $[T_A,T_A]$ acts transitively on the set of adjacency points between a red and a blue cell.
\end{corollary}

The proof consists of showing that an adjacency point between a red and a blue cell can be mapped to such a point lying on $\partial C_0$, which is easy using the previous lemma and the transitivity of $rist(C_0) \leq [T_A, T_A]$ (Theorem \ref{theorem_Airplane_rist_component}).

We will find additional transitivity properties of $[T_A, T_A]$ in Section \ref{section_E}. Now, with the aid of this Corollary, we are ready to prove the simplicity of $[T_A,T_A]$.

\begin{theorem}\label{theorem_TA_commutator_simple}
The commutator subgroup $[T_A, T_A]$ is simple.
\end{theorem}

This follows immediately from the following more general result:

\begin{proposition}\label{proposition_TA_commutator_simple}
The commutator subgroup $[T_A, T_A]$ is the only nontrivial proper subgroup of $T_A$ that is normalized by $[T_A, T_A]$.
\end{proposition}

\begin{proof}
This proof shares the overall structure with the proof of Theorem 8.4 from \cite{Belk_2015}, which itself follows the basic outline for the proof of the simplicity of $T$, which in turn is based on the work \cite{Epstein} of Epstein on the simplicity of groups of diffeomorphisms. In particular, we use Epstein’s double commutator trick, together with an argument that $T_A$ is generated by elements of small support.

Let $N$ be a nontrivial subgroup of $T_A$ that is normalized by $[T_A, T_A]$. We wish to prove that $N = [T_A, T_A]$, and we will do so by proving that $N$ contains each of the generators of the commutator subgroup exhibited in Corollary \ref{corollary_commutator_TA_generators}.

Let $f$ be a non-trivial element of $N$. There must exist a small enough cell $\Delta$ of the Airplane limit space such that $\Delta$ and $f(\Delta)$ are disjoint. Then, for all subsets $I$ of $\Delta$, we have that $I$ and $f(I)$ are disjoint. We will now prove a property that holds for every such $I$; we denote this property by ($\star$), and we will later use it twice.

Consider any two elements $g$ and $h$ of $[T_A, T_A]$ with support contained in $I$. The conjugate $f \circ g^{-1} \circ f^{-1}$ has support in $f(I)$, so $[g, f] = g \circ f \circ g^{-1} \circ f^{-1}$ has support in $I \cup f(I)$ and agrees with $g$ on $I$. Since $h$ has support in $I$, it follows that
\[ \big[[g,f],h\big] = [g,h]. \]
Since $N$ is normalized by $[T_A,T_A]$ by hypothesis, the double commutator on the left must be an element of $N$, so we have proved that 
\begin{gather}
[g, h] \in N \text{, for all } g, h \in [T_A, T_A] \text{ with support in } I. \tag{$\star$}
\end{gather}
We will use property ($\star$) for two distinct choices of the subset $I$ of $\Delta$. For each of these choices, we will conjugate ($\star$) by some properly chosen $k^{-1} \in [T_A,T_A]$; then, since $N$ is normalized by $[T_A,T_A]$, we will have that 
\[ [g, h] \in N \text{, for all } g, h \in [T_A, T_A] \text{ with support in } k(I). \]

\smallskip
\textit{First choice of $I \subseteq \Delta$.}\\
Consider a component $C$ that is contained in $\Delta$. Let $I$ be the union of the blue cell that departs from $C$ at angle $1/2$ and the two maximal red cells that make up the boundary of $C$. Call $p$ the point of $\partial C$ located at angle $0$. Figure \ref{fig_comm_A} shows an example of such $I$ and $p$.

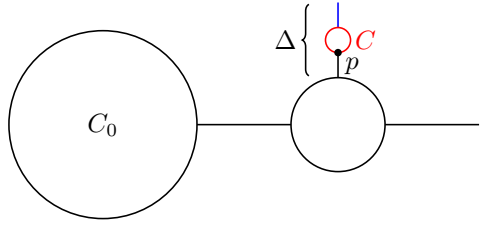
\begin{figure}\centering
\begin{tikzpicture}[scale=2.5]
\draw (0,0) node{$C_0$} circle (0.5);
\draw (0.5,0) -- (1,0);
\draw (1.5,0) -- (2,0);
\draw[blue] (1.25,0.65) -- (1.25,0.45);
\draw (1.25,0.45) -- (1.25,0.25);
\draw[red,fill=white] (1.25,0.45) node[right=0.1cm]{$C$} circle (0.0666);
\draw (1.25,0) circle (0.25);
\node at (1.25,0.3833)[circle,fill,inner sep=1]{};
\node at (1.325,0.315){$p$};
\draw[decoration={brace,mirror,raise=11pt},decorate] (1.25,0.64) -- node[left=12pt]{$\Delta$} (1.25,0.26);
\end{tikzpicture}
\caption{An example of the first choice of $I$, which is the union of the \textcolor{blue}{blue} cell and the two \textcolor{red}{red} ones. The figure is ``zoomed in'' on the right horizontal ray for better clarity.}
\label{fig_comm_A}
\end{figure}

Now, because of Corollary \ref{corollary_commutator_TA_transitive_points}, there exists an element $k \in [T_A, T_A]$ that maps $p$ to the point of adjacency between the central component and the left horizontal ray. Then $k(I)$ must be the union of the right horizontal blue cell and the two maximal red cells that make up the boundary of $C_0$. We denote this set by $K$.

Note that $I$ is included in $\Delta$, so the property ($\star$) holds. Then, conjugating ($\star$) by $k^{-1}$, we find that
\[ [g, h] \in N \text{, for all } g, h \in [T_A, T_A] \text{ with support in } K. \]

Next recall that $rist(C_0) \leq [T_A, T_A]$ (Remark \ref{remark_commutator_TA_contains}), so $[g, h] \in N$ for all $g, h \in rist(C_0)$ with support in $K$. Also recall from Theorem \ref{theorem_Airplane_rist_component} that $rist(C_0)$ is isomorphic to Thompson's group $T$; under this isomorphism, the set of those elements of $rist(C_0)$ whose support is included in $K$ corresponds to the stabilizer of $1/2$ in $T$, which is a copy of Thompson's group $F$. Since $F$ is not abelian, there exist at least two elements $g, h \in rist(C_0)$ with support in $K$ for which $[g, h]$ is nontrivial. Then $[g, h]$ belongs to both $N$ and $rist(C_0)$, so the intersection $N \cap rist(C_0)$ is a nontrivial normal subgroup of $rist(C_0)$. But $rist(C_0) \simeq T$ and $T$ is simple, so $N \cap rist(C_0) = rist(C_0)$, and therefore $rist(C_0) \leq N$.

\smallskip
\textit{Second choice of $I \subseteq \Delta$.}\\
Let $C$ be a component that is contained in $\Delta$. Let $R_1$ and $R_2$ be distinct rays that depart from $C$, and let $p_1$ be the point of adjacency between $C$ and $R_1$, and $p_2$ between $C$ and $R_2$. Call $I$ the union of $\partial C$, $\chi(R_1)$ and $\chi(R_2)$, where $\chi(R_i)$ are the blue cells associated with the rays $R_i$. Figure \ref{fig_comm_A1} shows an example of $I$, $p_1$ and $p_2$.

\begin{figure}\centering
\begin{tikzpicture}[scale=2.5]
\draw (0,0) node{$C_0$} circle (0.5);
\draw (0.5,0) -- (1,0);
\draw (1.5,0) -- (2,0);
\draw (1.25,0.65) -- (1.25,0.45);
\draw[blue] (1.39142, 0.308579) -- (1.25,0.45);
\draw[blue] (1.3184, 0.637939) -- (1.25,0.45);
\draw (1.25,0.45) -- (1.25,0.25);
\draw[red,fill=white] (1.25,0.45) node[left=0.08cm]{$C$} circle (0.0666);
\draw (1.25,0) circle (0.25);
\node at (1.2728, 0.512649)[circle,fill,inner sep=1]{}; \node at (1.4, 0.56){$p_1$};
\node at (1.29714, 0.402857)[circle,fill,inner sep=1]{}; \node at (1.435, 0.4){$p_2$};
\draw[decoration={brace,mirror,raise=11pt},decorate] (1.2,0.64) -- node[left=12pt]{$\Delta$} (1.2,0.26);
\end{tikzpicture}
\caption{An example of the second choice of $I$, which is the union of \textcolor{red}{$\partial C$} and two \textcolor{blue}{blue} cells departing from $C$. The figure is ``zoomed in'' on the right horizontal ray for better clarity.}
\label{fig_comm_A1}
\end{figure}
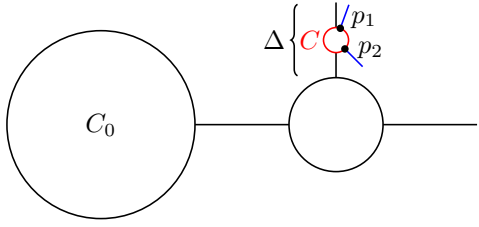

Because of Corollary \ref{corollary_commutator_TA_transitive_points}, there exists an element $k_1 \in [T_A,T_A]$ that maps $p_1$ to the point of adjacency between the central component and the right horizontal ray. Then $k_1(\partial C) = \partial C_0$, $k_1(\chi(R_1)) = \chi(R_0)$, $k_1(\chi(R_2))$ is some blue cell adjacent to $C_0$, and so $k_1(I)$ is the union of these three sets.

Now, since $T$ acts 2-transitively on the set of dyadic points of $S^1$ and because of Theorem \ref{theorem_Airplane_rist_component}, there exists an element $k_2 \in rist(C_0)$ that fixes $k_1(p_1)$ and maps $k_1(p_2)$ to the point of adjacency between the central component and the left horizontal ray. Then $k_2 \circ k_1 (I)$ is the union of $\partial C_0$ and the two horizontal central blue cells, and we denote this set by $K$. Also, since $rist(C_0) \leq [T_A,T_A]$ (Remark \ref{remark_commutator_TA_contains}), we have that $k_2 \in [T_A,T_A]$, therefore $k_2 \circ k_1 \in [T_A,T_A]$.

Note that $I$ is included in $\Delta$, so the property ($\star$) holds. Then, conjugating ($\star$) by $k_2 \circ k_1$, we find that
\[ [g, h] \in N \text{ for all } g, h \in [T_A, T_A] \text{ with support in } K. \]

Next, recall from Theorem \ref{theorem_Airplane_rist_hor} that the group $rist(Hor)$ is isomorphic to Thompson's group $F$. Under this isomorphism, the set of those elements of $rist(Hor)$ whose support is included in $K$ corresponds to the stabilizer of $1/2$. Since $1/2$ corresponds to the central component, this means exactly that
\begin{gather}
[g,h] \in N \text{ for all } g,h \in S(C_0), \tag{$\dagger$}
\end{gather}
where $S(C_0) := stab_{[rist(Hor), rist(Hor)]}(C_0)$ is the group consisting of those elements of $[rist(Hor), rist(Hor)]$ that fix the component $C_0$.

Now, since $[F,F]$ is transitive on the set of dyadic points of $(0,1)$, for each component $C'$ lying on $Hor$ there exists an element $l$ of $[rist(Hor), rist(Hor)] \leq [T_A,T_A]$ such that $l(C')=C_0$. By conjugating the group $S(C_0)$ by $l$, clearly we obtain exactly the group $S(C') := stab_{[rist(Hor), rist(Hor)]}(C')$ consisting of those elements of $[rist(Hor), rist(Hor)]$ that fix $C'$. Since we can do this for each component $C'$ lying on $Hor$, conjugating ($\dagger$) by each $l$ found this way we find that
\begin{gather*}
    [g,h] \in N \text{ for all } g,h \in S(C'), \text{ for all $C'$ lying on $Hor$}.
    \tag{$\#$}
\end{gather*}

We now prove that $[rist(Hor), rist(Hor)] \leq N$. First note that, since $rist(Hor) \simeq F$ and $[F, F] = F''$ (where $F''$ denotes the group $[[F,F],[F,F]]$), it suffices to prove that $rist(Hor)'' \leq N$. By definition, $rist(Hor)''$ is generated by the elements $[g,h]$ for $g,h \in [rist(Hor), rist(Hor)]$, so we only need to prove that these elements belong to $N$. If $g,h \in [rist(Hor), rist(Hor)]$, then they both act trivially around the extremes of $Hor$, so the intersection of their supports cannot be the entire $Hor$ and there must be some ``external enough'' component $C'$ lying on $Hor$ that is fixed by both $g$ and $h$, which means that $g,h \in S(C')$. Therefore, because of ($\#$), their commutator $[g,h]$ belongs to $N$. So $[rist(Hor), rist(Hor)] \leq N$.

\medskip
So far we have showed that:
\begin{enumerate}
    \item $rist(C_0) = \langle \beta, \gamma, \delta \rangle \leq N$;
    \item $[rist(Hor), rist(Hor)] = [\langle \alpha, \varepsilon \rangle, \langle \alpha, \varepsilon \rangle] \leq N$.
\end{enumerate}
With this in mind, we now show that $N$ contains each element of the set $\{ \beta, \gamma, \alpha^{\varepsilon^k}, \delta^{\varepsilon^k} \}$, which generates $[T_A, T_A]$ as seen in Corollary \ref{corollary_commutator_TA_generators}.

Since $rist(C_0) \leq N$, we have that $\beta, \gamma$ and $\delta$ belong to $N$. Consider $\delta^{\varepsilon^k}$, for any $k \in \mathbb{Z}$: with a direct computation, we note that $\delta \circ \beta$ has order three and $\delta = \delta^{-1}$, so $\delta = \delta^{-1} = \beta \circ \delta \circ \beta \circ \delta \circ \beta$, and then $\delta^{\varepsilon^k} = (\beta \circ \delta \circ \beta \circ \delta \circ \beta)^{\varepsilon^k}$. Since $\beta$ and $\varepsilon^k$ have disjoint supports, we have that $\beta^{\varepsilon^k} = \beta$. Therefore $\delta^{\varepsilon^k} = \beta \circ \delta^{\varepsilon^k} \circ \beta \circ \delta^{\varepsilon^k} \circ \beta = \beta \circ \beta^{\delta^{\varepsilon^k}} \circ \beta$, which belongs to $N$ because $\beta \in N$ and $\delta^{\varepsilon^k} \in [T_A, T_A] \trianglerighteq N$. Then $\delta^{\varepsilon^k} \in N$ for all $k \in \mathbb{Z}$, and we only need to prove that $\alpha^{\varepsilon^k} \in N$.

Recall that, as noted in Remark \ref{remark_commutator_TA_contains}), $\alpha = [\varepsilon, \delta] \circ [\varepsilon^{-1}, \alpha^{-2}]$, which can be seen with a direct computation. Note that $[\varepsilon, \delta] = \delta^{\varepsilon^{-1}} \circ \delta \in N$ and $[\varepsilon^{-1}, \alpha^{-2}] \in [rist(Hor), rist(Hor)] \leq N$, and so $\alpha \in N$. Finally, note that $\alpha^{\varepsilon^k} = \alpha \circ [\alpha^{-1}, \varepsilon^{-k}] \in \alpha \, [rist(Hor), rist(Hor)] \leq N$, for all $k \in \mathbb{Z}$.

Therefore $\beta, \gamma, \delta^{\varepsilon^k}$ and $\alpha^{\varepsilon^k}$ belong to $N$ for all $k \in \mathbb{Z}$, so $[T_A, T_A] \leq N$, which is what we needed to prove.
\end{proof}

\begin{question}
Having just proved that $[T_A, T_A]$ is finitely generated, it is natural to ask if it finitely presented or if it is not, as is the case for $T_B$ and $[T_B, T_B]$ (see \cite{witzar}). We have not investigated this question when writing this paper.
\end{question}

\subsection{The commutator subgroup is finitely generated}\label{subsection_commutator_TA_finitely_generated}

In this subsection we exhibit a finite set of generators for the commutator subgroup $[T_A, T_A]$.

\begin{remark}\label{remark_commutator_delta_epsilon}
It is easy to prove that $[\delta, \varepsilon]^k = [\delta, \varepsilon^k]$ using the fact that $\varepsilon^{-1}$ and $\varepsilon^{\delta}$ have disjoint supports.
\end{remark}

\begin{theorem}\label{theorem_commutator_TA_finitely_generated}
$[T_A, T_A] = \langle \alpha, \beta, \gamma, \delta, [\delta, \varepsilon], [\varepsilon^{-1}, \varepsilon^{-1} \circ \alpha] \rangle$.
\end{theorem}

\begin{proof}
Consider $G := \langle \alpha, \beta, \gamma, \delta, [\delta, \varepsilon], [\varepsilon^{-1}, \varepsilon^{-1} \circ \alpha] \rangle$, which is clearly a subgroup of $[T_A, T_A]$. We will show that $G$ is the entire commutator subgroup of $T_A$ by proving that it contains the infinite generating set $\{ \beta, \gamma, \alpha^{\varepsilon^k}, \delta^{\varepsilon^k} \}$ of Corollary \ref{corollary_commutator_TA_generators}.

Clearly both $\beta$ and $\gamma$ belong to $G$. Also note that, because of Remark \ref{remark_commutator_delta_epsilon}, we have that $\delta^{\varepsilon^k} = \delta \circ [\delta, \varepsilon^{-k}] = \delta \circ [\delta, \varepsilon]^{-k}$, which belongs to $G$. Hence, we only need to prove that $\alpha^{\varepsilon^k} \in G$ for all $k \in \mathbb{Z}$.

Let $F_0 := \langle [\delta, \varepsilon], [\varepsilon^{-1}, \varepsilon^{-1} \circ \alpha] \rangle \leq G$. The generators are depicted in Figure \ref{fig_gen_F0}. It is not hard to see that $F_0$ acts on $R_0$ as Thompson's group $F$ does on $[0,1]$, which is also the same way $rist(R_0)$ acts on $R_0$ (Proposition \ref{proposition_rist_ray}). We remark that the overall action of $F_0$ on the Airplane limit space is not trivial outside of $R_0$.

\begin{figure}\centering
\begin{subfigure}{\textwidth}\centering
\begin{tikzpicture}[scale=1.6]
    \draw (0,0) circle (0.5);
    \node at (0,0) {$C_0$};
    \draw (0.5,0) -- (0.65,0);
    \draw[blue] (0.75,0) circle (0.1);
    \draw[blue] (0.75,0.1) node[above]{\footnotesize A};
    \draw (0.85,0) -- (0.9,0);
    \draw[red] (0.925,0) circle (0.025);
    \draw (0.95,0) -- (1,0);
    \draw (1.25,0) circle (0.25);
    \draw (1.5,0) -- (2,0);
    
    \draw[-to] (2.175,0) -- (2.325,0);
    
    \draw (3,0) circle (0.5);
    \node at (3,0) {$C_0$};
    \draw (3.5,0) -- (3.55,0);
    \draw[blue] (3.575,0) circle (0.025);
    \draw[blue] (3.575,0.025) node[above]{\footnotesize A};
    \draw (3.6,0) -- (3.65,0);
    \draw[red] (3.75,0) circle (0.1);
    \draw (3.85,0) -- (4,0);
    \draw (4.25,0) circle (0.25);
    \draw (4.5,0) -- (5,0);
\end{tikzpicture}
\caption{The element $[\varepsilon^{-1}, \varepsilon^{-1} \circ \alpha]$. The left horizontal ray is fixed, hence it is omitted.}
\end{subfigure}\\
\vspace*{5pt}
\begin{subfigure}{\textwidth}\centering
\begin{tikzpicture}[scale=1]
    \draw (-2,0) -- (2,0);
    \draw[fill=white] (0,0) circle (0.5);
    \draw[Green,fill=white] (-1.25,0) circle (0.25);
    \draw[Plum,fill=white] (-0.75,0) circle (0.1);
    \draw[blue,fill=white] (1.25,0) circle (0.25);
    \draw[red,fill=white] (1.75,0) circle (0.1);
    \draw[blue] (1.25,0.25) node[above]{\footnotesize A};

    \draw[-to] (2.28,0) -- (2.52,0);
    
    \begin{scope}[xshift=4.8cm]
    \draw (-2,0) -- (2,0);
    \draw[fill=white] (0,0) circle (0.5);
    \draw[Green,fill=white] (-1.75,0) circle (0.1);
    \draw[Plum,fill=white] (-1.25,0) circle (0.25);
    \draw[blue,fill=white] (0.75,0) circle (0.1);
    \draw[red,fill=white] (1.25,0) circle (0.25);
    \draw[blue] (0.75,0.1) node[above]{\footnotesize A};
    \end{scope}
\end{tikzpicture}
\caption{The element $[\delta, \varepsilon]$.}
\end{subfigure}
\caption{The two generators of $F_0$}
\label{fig_gen_F0}
\end{figure}
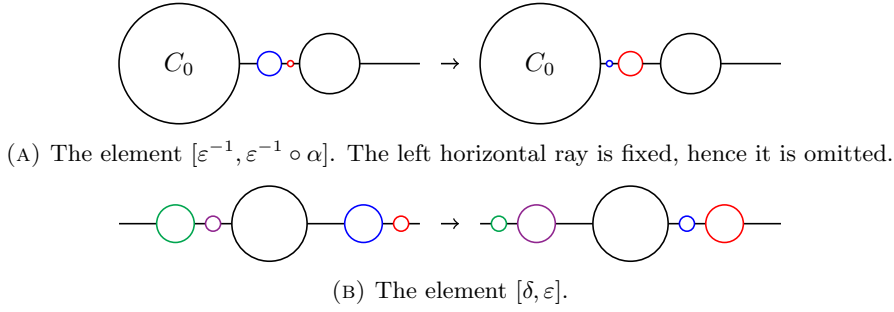

Let $r, l \in \mathcal{E}$ be the external extremes of the right and the left horizontal rays, respectively, and let $f$ be an element of $F_0$ such that $D_r(f) = 1$. Then $D_l (f) =1$ as well, because $f \in [T_A, T_A]$ and every other extremal derivative is clearly trivial. Now, since $f \in F_0$, there exists a finite product of elements chosen among $\{ [\delta, \varepsilon], [\varepsilon^{-1}, \varepsilon^{-1} \circ \alpha] \}$ that equals $f$. Since both these elements fix $r$ and $l$, and since all extremal derivatives of $[\varepsilon^{-1}, \varepsilon^{-1} \circ \alpha]$ are trivial, the value of $D_l (f)$ only depends on the total sum of the exponents of $[\delta, \varepsilon]$ in that product, so that total sum must be zero. Now, since $[\varepsilon^{-1}, \varepsilon^{-1} \circ \alpha]$ acts trivially on the left horizontal ray, the action of $f$ on that ray only depends on the total sum of exponents of $[\delta, \varepsilon]$. Then $f$ must act trivially on the left horizontal ray, which means that $f$ only acts on $R_0$. This proves that all elements $f \in F_0$ such that $D_r (f) = 1$ belong to $rist(R_0)$.

Now consider an element $g$ of $rist(R_0) \cap [T_A, T_A]$: since $F_0$ acts on $R_0$ as $rist(R_0)$ does, there exists an $f \in F_0$ such that $f$ acts on $R_0$ as $g$ does. Then note that $D_r (f) = D_r (g)$ (because $r$ is the external extreme of $R_0$), and $D_r (g) = 1$ because $g \in [T_A, T_A]$, so $D_r (f) = 1$. Then, as noted right above, $f$ belongs to $rist(R_0)$, and therefore it acts exactly as $g$ does on the entire $Hor$. Since both $f$ and $g$ are elements of $rist(Hor)$, they are entirely determined by their action on $Hor$, so $f = g$. Therefore we have that $rist(R_0) \cap [T_A, T_A] \leq F_0$.

Finally, let us prove that $\alpha^{\varepsilon^k} \in G$. Since $\alpha \in G$ and $\alpha^{\varepsilon^k} = [\varepsilon^{-k}, \alpha] \circ \alpha$, it suffices to prove that $[\varepsilon^{-k}, \alpha] \in G$. Note that $[\varepsilon^{-k}, \alpha]$ acts trivially on the left horizontal ray, and so it belongs to $rist(R_0)$. Then $[\varepsilon^{-k}, \alpha] \in [T_A, T_A] \cap rist(R_0) \leq F_0 \leq G$, and we are done.
\end{proof}

\section{\texorpdfstring{Thompson's group $T$ contains a copy of $T_A$}{Thompson's group T contains a copy of TA}}\label{section_circular}

In this section we show that $T$ contains an isomorphic copy of $T_A$.

Let $\mathcal{C(A)}$ be the replacement system whose set of colors is $\{blue, red \}$ and having base graph and replacement graphs as depicted in Figure \ref{fig_replacement_C(A)}. It is easy to see that this replacement system is such that its limit space exists (see Proposition 1.22 of \cite{belk2016rearrangement}), hence its rearrangement group exists, and we denote it by $C_A$. We will now prove that $T$ contains a copy of $C_A$ and then that $C_A$ contains a copy of $T_A$.

A replacement system is said to be \textbf{circular} if its base graph is a closed path and each of its replacement graphs consists solely of a path. Note that each expansion of a circular replacement system is always a closed path. Then it is not hard to prove that Thompson's group $T$ contains an isomorphic copy of any rearrangement group of a circular replacement system. It is clear that $\mathcal{C(A)}$ is a circular replacement system, so $T$ contains an isomorphic copy of $C_A$.

\begin{figure}\centering
\begin{subfigure}[b]{.3\textwidth}
\centering
\begin{tikzpicture}[scale=1.25]
    \draw[->-=.5, blue] (0:1) arc [radius=1, start angle=0, end angle=60];
    \draw[->-=.5, red] (60:1) arc [radius=1, start angle=60, end angle=120];
    \draw[->-=.5, blue] (120:1) arc [radius=1, start angle=120, end angle=180];
    \draw[->-=.5, blue] (180:1) arc [radius=1, start angle=180, end angle=240];
    \draw[->-=.5, red] (240:1) arc [radius=1, start angle=240, end angle=300];
    \draw[->-=.5, blue] (300:1) arc [radius=1, start angle=300, end angle=360];
    \node at (0:1) [circle, fill, inner sep = 1.25]{};
    \node at (60:1) [circle, fill, inner sep = 1.25]{};
    \node at (120:1) [circle, fill, inner sep = 1.25]{};
    \node at (180:1) [circle, fill, inner sep = 1.25]{};
    \node at (240:1) [circle, fill, inner sep = 1.25]{};
    \node at (300:1) [circle, fill, inner sep = 1.25]{};
\end{tikzpicture}
\caption{The base graph $\Gamma_C$.}
\end{subfigure}
\begin{subfigure}[b]{.6\textwidth}\centering
\begin{subfigure}{\textwidth}\centering
\begin{tikzpicture}[scale=.85]
    \draw[->-=.55, red] (-1.6,0) node[black, circle, fill, inner sep = 1.25]{} node[black, above]{$v_i$} -- (0,0) node[black, circle, fill, inner sep = 1.25]{} node[black, above]{$v_t$};
    
    \draw[-stealth] (0.5,0) -- (0.9,0);
    
    \draw[->-=.55,red] (1.4,0) node[black, circle, fill, inner sep = 1.25]{} node[black, above]{$v_i$} -- (2.4,0) node[black, circle, fill, inner sep = 1.25]{};
    \draw[->-=.55,blue] (2.4,0) node[black, circle, fill, inner sep = 1.25]{} -- (3.4,0) node[black, circle, fill, inner sep = 1.25]{};
    \draw[->-=.55,blue] (3.4,0) node[black, circle, fill, inner sep = 1.25]{} -- (4.4,0) node[black, circle, fill, inner sep = 1.25]{};
    \draw[->-=.55,red] (4.4,0) node[black, circle, fill, inner sep = 1.25]{} -- (5.4,0) node[black, circle, fill, inner sep = 1.25]{} node[black, above]{$v_t$};
\end{tikzpicture}
\end{subfigure}\\
\vspace{25pt}
\begin{subfigure}{\textwidth}\centering
\begin{tikzpicture}[scale=.85]
    \draw[->-=.55, blue] (-1.6,0) node[black, circle, fill, inner sep = 1.25]{} node[black, above]{$v_i$} -- (0,0) node[black, circle, fill, inner sep = 1.25]{} node[black, above]{$v_t$};
    
    \draw[-stealth] (0.5,0) -- (0.9,0);
    
    \draw[->-=.55,blue] (1.4,0) node[black, circle, fill, inner sep = 1.25]{} node[black, above]{$v_i$} -- (2.73333,0) node[black, circle, fill, inner sep = 1.25]{};
    \draw[->-=.55,red] (2.73333,0) node[black, circle, fill, inner sep = 1.25]{} -- (4.06667,0) node[black, circle, fill, inner sep = 1.25]{};
    \draw[->-=.55,blue] (4.06667,0) node[black, circle, fill, inner sep = 1.25]{} -- (5.4,0) node[black, circle, fill, inner sep = 1.25]{} node[black, above]{$v_t$};
\end{tikzpicture}
\end{subfigure}\\
\vspace{15pt}
\caption{The replacement rules.}
\end{subfigure}
\caption{The replacement system $\mathcal{C(A)}$.}
\label{fig_replacement_C(A)}
\end{figure}

Then we only need to prove that $C_A$ contains an isomorphic copy of $T_A$. We will first define an injective map $\Phi: \mathbb{E}_{\mathcal{A}} \to \mathbb{E}_{\mathcal{C(A)}}$, where $\mathbb{E}_{\mathcal{A}}$ is the set of all expansions of the Airplane replacement system $\mathcal{A}$ and $\mathbb{E}_{\mathcal{C(A)}}$ is the set of all expansions of $\mathcal{C(A)}$. Then we will use $\Phi$ to build an injective group morphism $\phi: T_A \to C_A$.

Note that the base graph $\Gamma_C$ of $\mathcal{C(A)}$ can be obtained from the base graph of $\mathcal{A}$ by splitting each vertex into two except for the two extremes, as shown in Figure \ref{fig_circularization}. We define $\Phi(\Gamma) := \Gamma_C$. Each red edge of $\Gamma$ is mapped to a red edge of $\Gamma_C$, while each blue edge is split into a pair of blue edges. Conversely, each red edge of $\Gamma_C$ descends from a unique red edge of $\Gamma$, while certain pairs of blue edges of $\Gamma_C$ descend from a unique blue edge of $\Gamma$. Informally, this gives a correspondence between the edges of $\Gamma$ and those of $\Gamma_C$ that is one-to-one between red edges and one-to-two between blue edges.

\begin{figure}\centering
\begin{tikzpicture}[scale=.7]
    \draw[->-=.34,->-=.68,gray!60] (0.5,0) to[out=90,in=240] (60:2);
    \draw[->-=.34,->-=.68,gray!60] (0.5,0) to[out=270,in=120] (300:2);
    \draw[->-=.34,->-=.68,gray!60] (-0.5,0) to[out=90,in=300] (120:2);
    \draw[->-=.34,->-=.68,gray!60] (-0.5,0) to[out=270,in=60] (240:2);
    
    \draw[->-=.5,blue,dashed] (-0.5,0) -- (-2,0);
    \draw[->-=.5,blue,dashed] (0.5,0) -- (2,0);
    \draw[->-=.5,red,dashed] (0.5,0) node[circle,fill,inner sep=1.25]{} to[out=90,in=90,looseness=1.7] (-0.5,0);
    \draw[->-=.5,red,dashed] (-0.5,0) node[black!65,circle,fill,inner sep=1.5]{} to[out=270,in=270,looseness=1.7] (0.5,0) node[black!65,circle,fill,inner sep=1.5]{};
    
    \draw[->-=.5, blue] (0:2) arc [radius=2, start angle=0, end angle=60];
    \draw[->-=.5, red] (60:2) arc [radius=2, start angle=60, end angle=120];
    \draw[->-=.5, blue] (120:2) arc [radius=2, start angle=120, end angle=180];
    \draw[->-=.5, blue] (180:2) arc [radius=2, start angle=180, end angle=240];
    \draw[->-=.5, red] (240:2) arc [radius=2, start angle=240, end angle=300];
    \draw[->-=.5, blue] (300:2) arc [radius=2, start angle=300, end angle=360];
    \node at (0:2) [circle, fill, inner sep = 1.5]{}; \node at (0:2) [right=.2cm]{$r$};
    \node at (60:2) [circle, fill, inner sep = 1.5]{};
    \node at (120:2) [circle, fill, inner sep = 1.5]{};
    \node at (180:2) [circle, fill, inner sep = 1.5]{}; \node at (180:2) [left=.2cm]{$l$};
    \node at (240:2) [circle, fill, inner sep = 1.5]{};
    \node at (300:2) [circle, fill, inner sep = 1.5]{};
\end{tikzpicture}
\caption{The base graph $\Gamma_C$ of $\mathcal{C(A)}$ obtained from the base graph $\Gamma$ of $\mathcal{A}$ (drawn with dashed lines).}
\label{fig_circularization}
\end{figure}
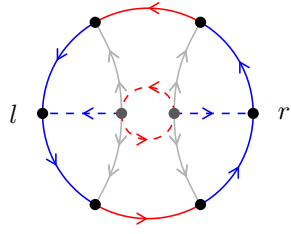

We now extend this correspondence to any expansion of the Airplane replacement system $\mathcal{A}$.
If $E$ is an expansion of $\mathcal{A}$ and $e$ is one of its edges, we denote by $E \lhd e$ the simple expansion obtained by replacing $e$ in $E$. Since each expansion $E$ of $\mathcal{A}$ is a finite sequence of simple expansions, it suffices to define $\Phi(E \lhd e)$ starting from $\Phi(E)$. We do this in the following way, distinguishing by the color of the edge $e$:
\begin{itemize}
    \item if $e$ is red, then $\Phi(E \lhd e)$ comes from $\Phi(E)$ by replacing its red edge corresponding to $e$ with the red replacement rule of $\mathcal{C(A)}$ (Figure \ref{fig_red_C(A)});
    \item if $e$ is blue, then $\Phi(E \lhd e)$ comes from $\Phi(E)$ by replacing the blue edges corresponding to $e$ with the blue replacement rule of $\mathcal{C(A)}$ (Figure \ref{fig_blue_C(A)}).
\end{itemize}
Note that we still have a correspondence between edges of $E$ and edges of $\Phi(E)$ that is one-to-one between red edges and ``one-to-two'' between blue edges. It is not hard to see that the order of simple expansions does not matter (which is, if $e_1, e_2 \in E^1$, then $\Phi(E \lhd e_1 \lhd e_2) = \Phi(E \lhd e_2 \lhd e_1)$), and that $\Phi$ is injective.

\begin{figure}\centering
\begin{tikzpicture}[scale=.7]
    \draw[->-=.34,->-=.68,gray!60] (0.5,0) to[out=90,in=240] (60:2);
    \draw[->-=.34,->-=.68,gray!60] (-0.5,0) to[out=90,in=300] (120:2);
    
    \draw[->-=.5,blue,dashed] (-0.5,0) -- (-2,0);
    \draw[->-=.5,blue,dashed] (0.5,0) -- (2,0);
    \draw[->-=.5,red,dashed] (0.5,0) node[circle,fill,inner sep=1.25]{} to[out=90,in=90,looseness=1.7] node[midway,above]{$e$} (-0.5,0);
    \draw[->-=.5,red,dashed] (-0.5,0) node[black!65,circle,fill,inner sep=1.5]{} to[out=270,in=270,looseness=1.7] (0.5,0) node[black!65,circle,fill,inner sep=1.5]{};
    
    \draw[->-=.5, blue] (0:2) arc [radius=2, start angle=0, end angle=60];
    \draw[->-=.5, red] (60:2) arc [radius=2, start angle=60, end angle=120];
    \draw[->-=.5, blue] (120:2) arc [radius=2, start angle=120, end angle=180];
    \draw[->-=.5, blue] (180:2) arc [radius=2, start angle=180, end angle=240];
    \draw[->-=.5, red] (240:2) arc [radius=2, start angle=240, end angle=300];
    \draw[->-=.5, blue] (300:2) arc [radius=2, start angle=300, end angle=360];
    \node at (0:2) [circle, fill, inner sep = 1.5]{};
    \node at (60:2) [circle, fill, inner sep = 1.5]{};
    \node at (120:2) [circle, fill, inner sep = 1.5]{};
    \node at (180:2) [circle, fill, inner sep = 1.5]{};
    \node at (240:2) [circle, fill, inner sep = 1.5]{};
    \node at (300:2) [circle, fill, inner sep = 1.5]{};
    
    \draw[-to,thick] (2.8,0) -- node[midway, above]{expanding \textcolor{red}{$e$}} (5.2,0);
    
    \begin{scope}[xshift=8cm]
    \draw[->-=.34,->-=.68,gray!60] (0.5,0) to[out=90,in=240] (60:2);
    \draw[->-=.34,->-=.68,gray!60] (-0.5,0) to[out=90,in=300] (120:2);
    \draw[->-=.5,gray!60] (0,1.2) to[out=90,in=270] (90:2);
    
    \draw[->-=.5,blue,dashed] (-0.5,0) -- (-2,0);
    \draw[->-=.5,blue,dashed] (0.5,0) -- (2,0);
    \draw[->-=.25,->-=.75,red,dashed] (0.5,0) node[circle,fill,inner sep=1.25]{} to[out=90,in=90,looseness=1.7] (-0.5,0);
    \draw[->-=.5,red,dashed] (-0.5,0) node[black!65,circle,fill,inner sep=1.5]{} to[out=270,in=270,looseness=1.7] (0.5,0) node[black!65,circle,fill,inner sep=1.5]{};
    \draw[->-=.5,blue,dashed] (0,0.5) node[black!65,circle,fill,inner sep=1.5]{} -- (0,1.2) node[black!65,circle,fill,inner sep=1.5]{};
    
    \draw[->-=.5, blue] (0:2) arc [radius=2, start angle=0, end angle=60];
    \draw[red] (60:2) arc [radius=2, start angle=60, end angle=75];
    \draw[blue] (75:2) arc [radius=2, start angle=75, end angle=90];
    \draw[blue] (90:2) arc [radius=2, start angle=90, end angle=105];
    \draw[red] (105:2) arc [radius=2, start angle=105, end angle=120];
    \draw[->-=.5, blue] (120:2) arc [radius=2, start angle=120, end angle=180];
    \draw[->-=.5, blue] (180:2) arc [radius=2, start angle=180, end angle=240];
    \draw[->-=.5, red] (240:2) arc [radius=2, start angle=240, end angle=300];
    \draw[->-=.5, blue] (300:2) arc [radius=2, start angle=300, end angle=360];
    \node at (0:2) [circle, fill, inner sep = 1.5]{};
    \node at (60:2) [circle, fill, inner sep = 1.5]{};
    \node at (75:2) [circle, fill, inner sep = 1.5]{};
    \node at (90:2) [circle, fill, inner sep = 1.5]{};
    \node at (105:2) [circle, fill, inner sep = 1.5]{};
    \node at (120:2) [circle, fill, inner sep = 1.5]{};
    \node at (180:2) [circle, fill, inner sep = 1.5]{};
    \node at (240:2) [circle, fill, inner sep = 1.5]{};
    \node at (300:2) [circle, fill, inner sep = 1.5]{};
    \end{scope}
\end{tikzpicture}
\caption{An example of $\Phi(E \lhd e)$ from $\Phi(E)$, where $e$ is a red edge.}
\label{fig_red_C(A)}
\end{figure}
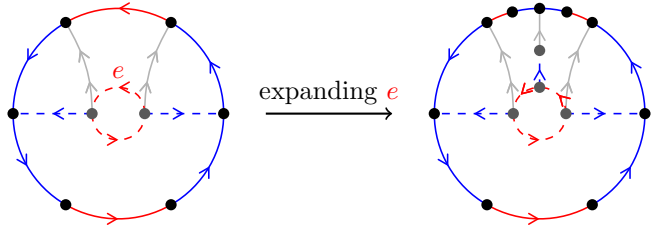

\begin{figure}\centering
\begin{tikzpicture}[scale=.7]
    \draw[->-=.34,->-=.68,gray!60] (0.5,0) to[out=90,in=240] (60:2);
    \draw[->-=.34,->-=.68,gray!60] (0.5,0) to[out=270,in=120] (300:2);
    
    \draw[->-=.5,blue,dashed] (-0.5,0) -- (-2,0);
    \draw[->-=.5,blue,dashed] (0.5,0) -- node[midway,above]{$e$} (2,0);
    \draw[->-=.5,red,dashed] (0.5,0) node[circle,fill,inner sep=1.25]{} to[out=90,in=90,looseness=1.7] (-0.5,0);
    \draw[->-=.5,red,dashed] (-0.5,0) node[black!65,circle,fill,inner sep=1.5]{} to[out=270,in=270,looseness=1.7] (0.5,0) node[black!65,circle,fill,inner sep=1.5]{};
    
    \draw[->-=.5, blue] (0:2) arc [radius=2, start angle=0, end angle=60];
    \draw[->-=.5, red] (60:2) arc [radius=2, start angle=60, end angle=120];
    \draw[->-=.5, blue] (120:2) arc [radius=2, start angle=120, end angle=180];
    \draw[->-=.5, blue] (180:2) arc [radius=2, start angle=180, end angle=240];
    \draw[->-=.5, red] (240:2) arc [radius=2, start angle=240, end angle=300];
    \draw[->-=.5, blue] (300:2) arc [radius=2, start angle=300, end angle=360];
    \node at (0:2) [circle, fill, inner sep = 1.5]{};
    \node at (60:2) [circle, fill, inner sep = 1.5]{};
    \node at (120:2) [circle, fill, inner sep = 1.5]{};
    \node at (180:2) [circle, fill, inner sep = 1.5]{};
    \node at (240:2) [circle, fill, inner sep = 1.5]{};
    \node at (300:2) [circle, fill, inner sep = 1.5]{};
    
    \draw[-to,thick] (2.8,0) -- node[midway, above]{expanding \textcolor{blue}{$e$}} (5.2,0);
    
    \begin{scope}[xshift=8cm]
    \draw[->-=.34,->-=.68,gray!60] (0.5,0) to[out=90,in=240] (60:2);
    \draw[->-=.34,->-=.68,gray!60] (0.5,0) to[out=270,in=120] (300:2);
    \draw[->-=.55,gray!60] (1,0) to[out=90,in=220,looseness=.8] (40:2);
    \draw[->-=.55,gray!60] (1,0) to[out=270,in=150,looseness=.8] (-40:2);
    \draw[gray!60] (1.5,0) to[out=90,in=200,looseness=.9] (20:2);
    \draw[gray!60] (1.5,0) to[out=270,in=160,looseness=.9] (-20:2);
    
    \draw[->-=.5,blue,dashed] (-0.5,0) -- (-2,0);
    \draw[red,dashed] (1.25,0) circle (.25);
    \draw[blue,dashed] (0.5,0) -- (1,0) node[black!65,circle,fill,inner sep=1.5]{};
    \draw[blue,dashed] (1.5,0) node[black!65,circle,fill,inner sep=1.5]{} -- (2,0);
    \draw[->-=.5,red,dashed] (0.5,0) node[circle,fill,inner sep=1.25]{} to[out=90,in=90,looseness=1.7] (-0.5,0);
    \draw[->-=.5,red,dashed] (-0.5,0) node[black!65,circle,fill,inner sep=1.5]{} to[out=270,in=270,looseness=1.7] (0.5,0) node[black!65,circle,fill,inner sep=1.5]{};
    
    \draw[blue] (0:2) arc [radius=2, start angle=0, end angle=20];
    \draw[red] (20:2) arc [radius=2, start angle=20, end angle=40];
    \draw[blue] (40:2) arc [radius=2, start angle=40, end angle=60];
    \draw[->-=.5, red] (60:2) arc [radius=2, start angle=60, end angle=120];
    \draw[->-=.5, blue] (120:2) arc [radius=2, start angle=120, end angle=180];
    \draw[->-=.5, blue] (180:2) arc [radius=2, start angle=180, end angle=240];
    \draw[->-=.5, red] (240:2) arc [radius=2, start angle=240, end angle=300];
    \draw[blue] (300:2) arc [radius=2, start angle=300, end angle=320];
    \draw[red] (320:2) arc [radius=2, start angle=320, end angle=340];
    \draw[blue] (340:2) arc [radius=2, start angle=340, end angle=360];
    \node at (0:2) [circle, fill, inner sep = 1.5]{};
    \node at (20:2) [circle, fill, inner sep = 1.5]{};
    \node at (40:2) [circle, fill, inner sep = 1.5]{};
    \node at (60:2) [circle, fill, inner sep = 1.5]{};
    \node at (120:2) [circle, fill, inner sep = 1.5]{};
    \node at (180:2) [circle, fill, inner sep = 1.5]{};
    \node at (240:2) [circle, fill, inner sep = 1.5]{};
    \node at (300:2) [circle, fill, inner sep = 1.5]{};
    \node at (320:2) [circle, fill, inner sep = 1.5]{};
    \node at (340:2) [circle, fill, inner sep = 1.5]{};
    \end{scope}
\end{tikzpicture}
\caption{An example of $\Phi(E \lhd e)$ from $\Phi(E)$, where $e$ is a blue edge.}
\label{fig_blue_C(A)}
\end{figure}
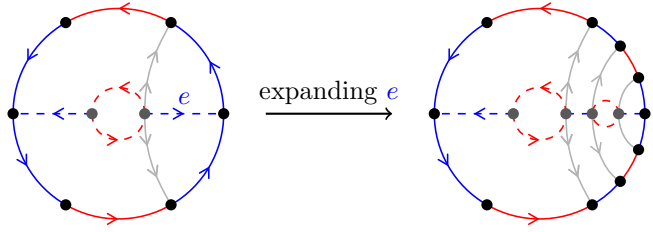

Now consider the map $\phi: T_A \to C_A$ defined in the following way. If $f$ is an element of $T_A$, let $D \to R$ be a graph pair diagram for $f$. Then $\phi(f)$ is the element of $C_A$ that is represented by the graph pair diagram $\Phi(D) \to \Phi(R)$, where the graph isomorphism is defined by the correspondence between red edges of $D$ (respectively $R$) and red edges of $\Phi(D)$ (respectively $\Phi(R)$) and between blue edges of $D$ (respectively $R$) and pairs of blue edges of $\Phi(D)$ (respectively $\Phi(R)$), which is, if $D \to R$ maps the edge $e_D$ to the edge $e_R$, then $\Phi(D) \to \Phi(R)$ maps the edge(s) corresponding to $e_D$ to the edge(s) corresponding to $e_R$. This definition does not depend on the graph pair diagram chosen to represent $f$, since expansions in $\mathcal{A}$ correspond to expansions in $\mathcal{C}_A$ (one-to one or ``one-to-two'', depending on the color of the edge).

Now, let $f$ and $g$ be rearrangements of the Airplane limit space, and consider their respective graph pair diagrams $D_f \to R_f$ and $D_g \to R_g$ such that $D_f = R_g$. Then their composition $f \circ g$ is represented by the graph pair diagram $D_g \to R_f$, so $\phi(f \circ g)$ is represented by $\Phi(D_g) \to \Phi(R_f)$. We also have that $\Phi(D_f) = \Phi(R_g)$, which means that the domain graph of $\Phi(D_f) \to \Phi(R_f)$ is the same as the range graph of $\Phi(D_g) \to \Phi(R_g)$. Hence, their composition $\phi(f) \circ \phi(g)$ is represented by $\Phi(D_g) \to \Phi(R_f)$. It is easy to see that these graph pair diagrams share the same graph isomorphism, thus $\phi$ is a group morphism.

Also, the kernel of $\phi$ is trivial. Indeed, if $f \in T_A$ is such that $\phi(f)$ is trivial, then consider a graph pair diagram for $f$ and let $e$ be a red edge of the domain graph: the corresponding edge in the domain graph of $\phi(f)$ must be fixed by $\phi(f)$, hence $e$ is fixed by $f$. The same holds for blue edges. Then $f$ must be trivial too.

Now, since $\phi$ is an injective morphism, we have found that $C_A$ contains an isomorphic copy of $T_A$. Since we have previously seen that $T$ contains an isomorphic copy of $C_A$, we can finally conclude that:
\begin{theorem}\label{theorem_T_contains_TA}
Thompson's group $T$ contains an isomorphic copy of $T_A$.
\end{theorem}

The replacement system $\mathcal{C(A)}$ (Figure \ref{fig_replacement_C(A)}) and the way in which it is related to the Airplane limit space (as depicted in Figure \ref{fig_circularization}, \ref{fig_red_C(A)} and \ref{fig_blue_C(A)}) are inspired by the original study of the Basilica rearrangement group $T_B$ in \cite{Belk_2015}: there, the group is not defined by the action on the Basilica Julia set, but instead by its action on the lamination of the fractal. The lamination is essentially the ``explosion'' of the Basilica on $S^1$ and it expresses the canonical way in which the Basilica Julia set is a quotient of the circle. Here we essentially did the same for the Airplane, and it would be interesting to see how and when this can be replicated to other fractals. It must be noted, however, that there are fractals for which similar arguments would not provide an embedding into $T$ such as the Vicsek fractal (Example 1.12 of \cite{belk2016rearrangement}), for its rearrangement group contains finite subgroups that are not cyclic.

\section{\texorpdfstring{Rearrangements with trivial extremal derivatives}{Rearrangements with trivial extremal derivatives}}\label{section_E}

In this section we study the following subgroup of $[T_A, T_A]$:
\[ E := \{ f \in T_A \:|\: D_p (f) = 1, \, \forall p \in \mathcal{E} \}. \]
We show that $E$ is not finitely generated and then we study its transitivity properties. In particular, we will see that its action is 2-transitive on the set of components, and so is the action of both $T_A$ and $[T_A, T_A]$.

\subsection{\texorpdfstring{$E$ is not finitely generated}{E is not finitely generated}}

We say that a blue edge is \textbf{external} if one of its two vertices corresponds to an element of $\mathcal{E}$ in the limit space. With this in mind, we say that a blue cell is \textbf{external} if it is generated by an external blue edge.

Recall the definition of length of blue edges and cells given at page \pageref{length}. Recall from Subsection \ref{subsection_component_paths} that each component is uniquely identified on the ray on which it lies by a dyadic number in $(0,1)$. For each $n \in \mathbb{N} \setminus \{0\}$, let $C(n)$ be the set of those components that lie at position $1 - \frac{1}{2^n}$ on some ray. For each $n \in \mathbb{N}\setminus\{0\}$, we define $E_n$ to be the subset of $T_A$ consisting of all rearrangements $f$ such that:
\begin{itemize}
    \item $f$ acts by permutation on the set $C(n)$;
    \item $f$ acts canonically on the red cells corresponding to the boundaries of components in $C(n)$;
    \item $f$ acts canonically on the external blue cells of length $1 - \frac{1}{2^n}$.
\end{itemize} Intuitively, these rearrangements are the ones that act ``rigidly'' on anything that lies ``beyond'' a component of $C(n)$, while they act without further constraints on the inner part of the Airplane delimited by the components of $C(n)$. Figure \ref{fig_En}, for $n=1$ and $n=2$ respectively, exhibits: in \textcolor{red}{red} the component lying on the right horizontal ray that belongs to $C(n)$ and in \textcolor{blue}{blue} the corresponding set on which $E_n$ acts canonically. Keep in mind that $E_n$ acts canonically on similar sets on every ray.

\begin{figure}\centering
\begin{subfigure}{.4\textwidth}\centering
\begin{tikzpicture}[scale=1.1]
    \draw (-2,0) -- (-0.5,0);
    \draw (0,0) circle (0.5);
    \draw[red,fill=red!25] (1.25,0) circle (0.25);
    \draw (0.5,0) -- (1,0);
    \draw[blue] (1.5,0) -- (2,0);
\end{tikzpicture}
\caption{$n=1$.}
\end{subfigure}
\begin{subfigure}{.4\textwidth}\centering
\begin{tikzpicture}[scale=1.1]
    \draw (-2,0) -- (-0.5,0);
    \draw (0,0) circle (0.5);
    \draw (0.5,0) -- (1,0);
    \draw (1.25,0) circle (0.25);
    \draw[red,fill=red!25] (1.75,0) circle (0.1);
    \draw (1.5,0) -- (1.65,0);
    \draw[blue] (1.85,0) -- (2,0);
\end{tikzpicture}
\caption{$n=2$.}
\end{subfigure}
\caption{The \textcolor{red}{component} of $C(n)$ that lies on the right horizontal ray, for $n=1,2$. The groups $E_n$ act ``rigidly'' on the entire colored subset, and the same happens for each ray of the Airplane.}
\label{fig_En}
\end{figure}
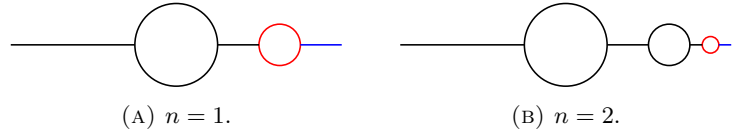

Note that all $E_n$ are groups, and they are all proper subgroups of $E$, because they preserve lengths around each extreme. Also, clearly $E_n \leq E_{n+1}$ for all $n \in \mathbb{N}\setminus\{0\}$. Moreover, it is easy to see that
\[ \bigcup\limits_{n>0} E_n = E. \]
Then, since $E$ is an ascending union of proper subgroups, we have that:

\begin{proposition}\label{proposition_E_finitely_generated}
The group $E$ is not finitely generated.
\end{proposition}

Recall that a subgroup $H$ of a finitely generated group $G$ is finitely generated too if the index $|G:H|$ is finite (see Corollary 9.2 at page 70 of \cite{Bogopolski}). Then we immediately have the following consequence of the previous proposition:
\begin{corollary}
The index of $E$ in $[T_A, T_A]$ is infinite.
\end{corollary}

\subsection{\texorpdfstring{Transitivity properties of $E$}{Transitivity properties of E}}\label{subsection_E_transitive}

If $C$ and $C'$ are two distinct components, we can define, similarly to what we have done in Subsection \ref{subsection_component_paths}, a unique component path that reaches $C'$ starting from $C$ instead of $C_0$. Let $n \geq 2$ and let $C_1, C_2, \dots, C_n$ be distinct components of the Airplane limit space. We say that these components are \textbf{aligned} if there exist two of them, $C_i$ and $C_j$, such that the component path between $C_i$ and $C_j$ travels through each component $C_k$. We then say that $C_i$ and $C_j$ are the \textbf{extremes} of these aligned components. It is not hard to see that rearrangements of the Airplane limit space preserve alignment, which is, if $f \in T_A$, then $C_1, C_2, \dots, C_n$ are aligned if and only if $f(C_1), f(C_2), \dots, f(C_n)$ are.

Note that, if $C_1, C_2, \dots, C_n$ are aligned, we can rename them so that $C_1$ and $C_n$ are the extremes and the components $C_i$ are ordered from the first to the last traveled from $C_1$ to $C_n$. Once we have renamed them in this fashion, we say that $(C_1, C_2, \dots, C_n)$ is an \textbf{ordered $n$-tuple of aligned components}.

It is clear that $E$ contains both $rist(C_0)$ and $[rist(Hor), rist(Hor)]$. Recall that $rist(C_0)$ acts transitively on the set of central rays (Theorem \ref{theorem_Airplane_rist_component}). Also, it is well-known that $[F,F]$ acts transitively on the set of ordered $n$-tuples of dyadic points of $(0,1)$. Hence, because of Theorem \ref{theorem_Airplane_rist_hor}, the group $[rist(Hor), rist(Hor)]$ acts transitively on the set of ordered $n$-tuples of aligned components that lie on $Hor$. Using these transitivity properties of $rist(C_0)$ and $[rist(Hor), rist(Hor)]$, it is not hard to prove the following result:

\begin{proposition}
For all natural $n>0$, the group $E$ acts transitively on the set of ordered $n$-tuples of aligned components.
\end{proposition}

Since any two components are aligned, the case $n=2$ is interesting on its own:
\begin{corollary}
The group $E$ acts 2-transitively on the set of components. In particular, the groups $[T_A, T_A]$ and $T_A$ act 2-transitively on the set of components.
\end{corollary}

\begin{remark}
The group $T_A$ does \textbf{not} act 3-transitively on the set of components, therefore neither $[T_A, T_A]$ nor $E$ do. Indeed, consider three components that are not aligned (for example, any three components lying on three distinct central rays) and three components that are aligned (for example, any three components that lie on $Hor$): since rearrangements preserve alignment, there cannot exist any rearrangement that maps those first three components to the others.
\end{remark}

\section{\texorpdfstring{A copy of $T_B$ in $T_A$}{A copy of TB in TA}}\label{section_TB_in_TA}

In this section we exhibit an isomorphic copy of the Basilica rearrangement group $T_B$ that is contained in $T_A$. Note that it is already clear that $T_A$ contains an isomorphic copy of $T_B$: indeed, the subgroup $rist(C_0)$ of $T_A$ is isomorphic to $T$ (Theorem \ref{theorem_Airplane_rist_component}), and $T$ contains a copy of $T_B$ (proved in \cite{Belk_2015}). This reasoning, however, does not tell us much about the nature of this copy of $T_B$ in $T_A$. Here we instead find a natural copy of $T_B$ in $T_A$ by specifying its generators. In particular, this section focuses on proving the following result:

\begin{theorem}
$T_B \simeq \langle \alpha, \beta, \gamma, \delta \rangle \leq T_A$.
\end{theorem}

We start by defining a non-directed graph $\Gamma_B$ in the following way:
\begin{itemize}
    \item its vertices are the components of $B$;
    \item its edges link pairs of adjacent components.
\end{itemize}
Note then that $\Gamma_B$ is a rooted regular tree of infinite degree, where the root is the central component. since rearrangements of the Basilica limit space permute components of $B$ and preserve their adjacency, $T_B$ acts faithfully on $\Gamma_B$ by graph automorphisms. 
This means that there exists an injective morphism $\phi_B: T_B \to G$, where $G := Aut(\Gamma_B)$.

We now define a similar tree for the Airplane limit space by considering the components whose component paths have the following form:
\[ \left(\left(\theta_1,\frac{2^{k_1}-1}{2^{k_1}}\right), \dots, \left(\theta_n,\frac{2^{k_n}-1}{2^{k_n}}\right) \right), \]
for some naturals $n>0$ and $k_i>0$, and for some dyadic $\theta_i \in [0,1)$.
We denote by $\mathfrak{C}$ the set consisting of all of these components, along with the central component. Figure \ref{fig_mathrak_C} depicts examples of components that belong to $\mathfrak{C}$ (in \textcolor{blue}{blue}) and that do not (in \textcolor{red}{red}). Note that these components are the ones that can be obtained by halving any central ray (which locates certain central components), then halving the external remaining part of that ray or halving any new ray departing from the located component, and so on until one stops at the newfound component.

\begin{figure}[t]\centering
\begin{tikzpicture}[scale=1.8]
\draw (-2,0) -- (2,0);
\draw (1.25,0.65) -- (1.25,0);
\draw (-0.75,0) -- (-0.75,0.383);
\draw[blue,fill=blue!25] (0,0) circle (0.5);
\draw[blue] (0,0) node{$C_0$};
\draw[blue,fill=blue!25] (-1.25,0) circle (0.25);
\draw[blue] (-1.25,.25) node[above]{$C_1$};
\draw[blue,fill=blue!25] (1.25,0) circle (0.25);
\draw[blue] (1.42,-.08) node[below right]{$C_3$};
\draw[red,fill=red!25] (0.75,0) circle (0.083);
\draw[red,fill=red!25] (-0.75,0) circle (0.083);
\draw[blue,fill=blue!25] (-1.75,0) circle (0.083);
\draw[blue] (-1.75,.083) node[above]{$C_2$};
\draw[red,fill=red!25] (-0.75,0.233) circle (0.05);
\draw[blue,fill=blue!25] (1.25,0.45) circle (0.0666);
\draw[blue] (1.3166,.45) node[right]{$C_4$};
\draw[blue,fill=blue!25] (-1.9165,0) circle (0.0275);
\end{tikzpicture}
\caption{The components colored in \textcolor{blue}{blue} belong to $\mathfrak{C}$, while those in \textcolor{red}{red} do not.}
\label{fig_mathrak_C}
\end{figure}
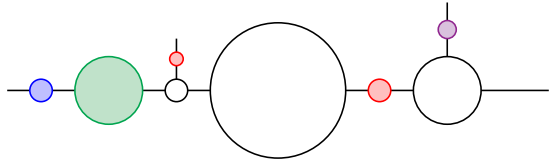

Note that these components have a natural concept of adjacency: the component path representing a component $C$ that belongs to $\mathfrak{C}$ travels through a finite amount of other components in $\mathfrak{C}$, the last of which is said to be $\mathfrak{C}$-adjacent to $C$. For example, in Figure \ref{fig_mathrak_C} the component $C_1$ is adjacent to both $C_2$ and $C_0$, but it is not adjacent to $C_3$ nor $C_4$ (which are themselves adjacent to each other); moreover, $C_2$ and $C_0$ are not adjacent (despite being related).

Now, let $\Gamma_A$ be the graph defined in the following way:
\begin{itemize}
    \item its vertices are the elements of $\mathfrak{C}$;
    \item its edges link pairs of $\mathfrak{C}$-adjacent components.
\end{itemize}
Note that $\Gamma_A$ is a rooted regular tree of infinite degree, hence it is isomorphic to $\Gamma_B$, and so the previously defined group $G=Aut(\Gamma_B)$ is isomorphic to $Aut(\Gamma_A)$. We can then refer to $G$ as the automorphism group of both $\Gamma_B$ and $\Gamma_A$.

Let $H := \langle \alpha, \beta, \gamma, \delta \rangle \leq T_A$. It is easy to see that the action of $H$ maps elements of $\mathfrak{C}$ to elements of $\mathfrak{C}$ and preserves $\mathfrak{C}$-adjacency, hence it is by graph automorphisms on $\Gamma_A$. It can also be showed that the action is faithful: it suffices to prove that, if $h \in H$ fixes each component in $\mathfrak{C}$, then it must fix each component $C$ of the Airplane. Now, since the action of $H$ on $\Gamma_A$ is faithful and it is by graph automorphisms, there exists an injective group morphism $\phi_A: H \to G$. Finally, it is easy to note that the four generators of the group $T_B$ (depicted in Figure \ref{fig_TB_generators}) act on $\Gamma_B$ exactly as $\alpha, \beta, \gamma$ and $\delta$ do on $\Gamma_A$, which proves that $T_B \simeq \langle \alpha, \beta, \gamma, \delta \rangle \leq T_A$.

\printbibliography[heading=bibintoc]

\end{document}